\documentclass[12pt]{article}

\usepackage[a4paper]{anysize}\marginsize{2.5cm}{2.5cm}{1cm}{1.8cm}
\pdfpagewidth=\paperwidth \pdfpageheight=\paperheight
\usepackage[latin1]{inputenc}
\usepackage{amssymb}

\usepackage{pgf,tikz,pgfplots}

\makeatletter

\renewcommand\@seccntformat[1]{\csname the#1\endcsname.\enspace}
\renewcommand\@begintheorem[2]{\trivlist\item[\hskip\labelsep{\bfseries#1 #2.}]\it}
\renewcommand\@opargbegintheorem[3]{\trivlist\item[\hskip\labelsep{\bfseries#1 #2}] {\bfseries(#3).}\enspace\it\ignorespaces}
\renewenvironment{abstract}{\begin{quote}\hrulefill\par\footnotesize\textbf{\abstractname.}}{\par\vskip-0.5\baselineskip\hrulefill\end{quote}}
\makeatother

\newtheorem{introtheorem}{Theorem}  
\newtheorem{introcorollary}[introtheorem]{Corollary}  

\newtheorem{thm}{Theorem}[section]
\newtheorem{theorem}[thm]{Theorem}
\newtheorem{lemma}[thm]{Lemma}
\newtheorem{proposition}[thm]{Proposition}
\newtheorem{corollary}[thm]{Corollary}
\newtheorem{conjecture}[thm]{Conjecture}

\newcommand\mkthm[2]{\newenvironment{#1}{\begin{#2}\rm}{\end{#2}}}
 \mkthm{definition}{tdefinition}
         \mkthm{remark}{tremark}
       \mkthm{example}{texample}
     \mkthm{question}{tquestion}

\newtheorem{thevarthm}[thm]{\varthmname}

\newenvironment{varthm*}[1]{\trivlist\item[]{\bf #1.}\it}{\endtrivlist}

\newenvironment{introdefinition}{\begin{varthm*}{Definition}\rm}{\end{varthm*}}

\newenvironment{proof}[1][Proof]{\trivlist\item[\hskip\labelsep{\textit{#1.}}]}{\hspace*{\fill}$\Box$\endtrivlist}

\renewcommand\O{\mathcal O}

\newcommand\engqq[1]{``#1''}
\renewcommand\emptyset{\varnothing}  
\renewcommand\ge{\geqslant}  
\renewcommand\le{\leqslant}  

\newcommand\keywords[1]{{\renewcommand\thefootnote{}\footnotetext{\emph{Keywords:} #1.}}}
\newcommand\subclass[1]{{\renewcommand\thefootnote{}\footnotetext{\emph{Mathematics Subject Classification (2010):} #1.}}}

\newcommand\Q{\mathbb Q}
\newcommand\R{\mathbb R}
\newcommand\Z{\mathbb Z}
\newcommand\N{\mathbb N}

\newcommand\be[1][@{\;}r@{\;}c@{\;}l@{\;}l@{\;}]{$$\everymath{\displaystyle}\renewcommand\arraystretch{1.2}\begin{array}{#1}}
\newcommand\ee{\end{array}$$}

\newcommand\midtext[1]{\quad\mbox{#1}\quad}
\newcommand\righttext[1]{\qquad\mbox{#1}}
\newcommand\matr[1]{\left(\begin{array}{*{20}{c}} #1 \end{array}\right)}
\newcommand\inverse{^{\smash-\mkern-1mu1}}
\newcommand\tfrac[2]{{\textstyle\frac{#1}{#2}}}  
\newcommand\compact{\itemsep=0cm \parskip=0cm}
\newcommand\set[1]{\left\{#1\right\}}
\newcommand\with{\,\,\vrule\,\,}

\newcommand\q[3]{\frac{#1\cdot #2}{\mult_{#3}(#2)}}  
\newcommand\tmod[1]{\;({\rm mod}~#1)}
\newenvironment{bycases}{\left\{\begin{array}{@{}l@{\quad}l}}{\end{array}\right.}
\newcommand\isom{\simeq}  
\newcommand\eqnref[1]{(\ref{#1})}

\newcommand\newop[2]{\newcommand#1{\mathop{\rm #2}\nolimits}}

\newop\Bl{Bl}
\newop\Amp{Amp}
\newop\Aut{Aut}
\newop\Nef{Nef}
\newop\NS{NS}
\newop\mult{mult}
\newop\End{End}
\newcommand\Endsym{\End^{\rm sym}}

\newcommand\eps{\varepsilon}

\begin{document}

   \title{Seshadri constants on principally polarized \\ abelian surfaces with real multiplication}
   \author{\normalsize Thomas Bauer, Maximilian Schmidt}
   \date{\normalsize September 7, 2021}
   \maketitle
   \thispagestyle{empty}
   \keywords{abelian surface, Seshadri constant, real multiplication, Cantor function}
   \subclass{14C20, 14K12, 26A30}

\begin{abstract}
   Seshadri constants on abelian surfaces are fully
   understood in the case of Picard number one.
   Little is known so far for simple abelian surfaces of higher
   Picard number. In this paper we investigate
   principally polarized abelian surfaces with real
   multiplication. They are of Picard number two and might be
   considered the
   next natural case to be studied.
   The challenge is
   to not only determine the Seshadri constants of individual line bundles, but to
   understand the whole
   \emph{Seshadri function} on these surfaces.
   Our results show on the one hand that
   this function is
   surprisingly complex:
   On surfaces with real multiplication in $\Z[\sqrt e]$ it consists
   of linear segments that are never adjacent to each other -- it behaves like the Cantor function.
   On the other hand,
   we prove that
   the Seshadri function
   is invariant
   under an infinite group of automorphisms, which shows that
   it
   does have interesting regular behavior globally.
\end{abstract}


\section*{Introduction}

   The purpose of this paper is to contribute to the study
   of Seshadri constants on abelian surfaces.
   Recall that
   for an ample line bundle $L$
   on a smooth projective variety $X$, the \emph{Seshadri constant}
   of $L$ at a point $x\in X$ is by definition
   the real number
   \be
      \eps(L,x)=\inf\set{\q LCx\with C\mbox{ irreducible curve through } x}
      \,.
   \ee
   On
   abelian varieties, where
   this invariant is independent of
   the chosen point $x$, we write simply $\eps(L)$.
   Seshadri constants are highly interesting invariants
   for numerous reasons: They
   are related
   to minimal period lengths \cite{Lazarsfeld:periods,Bauer:periods},
   to syzygies \cite{LPP:positivity-syzygies,Kuronya-Lozovanu:Reider-type-theorem},
   and they govern quite generally the geometry of linear series in many
   respects
   \cite{Ein-Kuechle-Lazarsfeld,ELMNP:restricted-volumes}
   (we refer to \cite[Chapt.~5]{Lazarsfeld:PAG} and \cite{Bauer-et-al:primer}
   for more background on Seshadri constants).

   On abelian surfaces, Seshadri constants
   are fully understood in the case
   of Picard number $\rho=1$ \cite{Bauer:sesh-alg-sf}.
   For $\rho>1$, only self-products of elliptic curves have been
   studied, while the important case of simple abelian surfaces is
   completely unexplored so far.
   In contrast to the case of $\rho=1$,
   the challenge on these surface is not only to determine
   the Seshadri constant of one ample line bundle, but
   to understand the behavior of the
   \emph{Seshadri function},
   \be
      \eps:\Amp(X)\to\R, \quad L\mapsto \eps(L)
      \,,
   \ee
   which associates
   to each ample line bundle its Seshadri constant.
   To our knowledge, there are --
   also beyond abelian surfaces --
   hardly any cases where
   this function is known explicitly,
   the exception being
   certain self-products $E\times E$ of elliptic curves \cite{Bauer-Schulz}.
   In general, the Seshadri function
   of an abelian variety
   is known to be concave and continuous
   \cite[Prop.~3.1]{Bauer-Schulz},
   but
   at present it is unclear what kind of behavior to expect
   beyond these basic properties.

   We attack this problem on abelian surfaces of Picard number
   $\rho=2$, which seems to be the natural next case to
   investigate.
   As the Seshadri
   function is homogeneous, it is completely determined by its
   values on a cross-section of $\Amp(X)$.
   So, when $\Amp(X)$ is two-dimensional, we may consider it
   as a function $\eps:I\to\R$ on an interval $I\subset\R$.
   We always take this point of view when we speak of
   the Seshadri function.

   For clarity of exposition let us introduce a piece of
   terminology:

\begin{introdefinition}
   Let $I\subset\R$ be an interval. A function
   $f:I\to\R$ is called \emph{broken linear},
   if it is continuous and there is
   a non empty and nowhere dense
   subset $M\subset I$ such that
   the following holds:
   \begin{itemize}\compact
   \item[(i)]
      Around every point of $I\setminus M$ there is an open
      interval, contained in
      $I\setminus M$,
      on which $f$ is linear.

   \item[(ii)]
      If $I_1$ and $I_2$ are maximal open subintervals of $I$ on which
      $f$ is linear, then $I_1$
      and $I_2$ are contained in $I\setminus M$, and
      $I_1$ and $I_2$ are not adjacent to each other
      (i.e., an endpoint of $I_1$ is never an endpoint of $I_2$).

   \end{itemize}
\end{introdefinition}
   Note that these conditions imply that $M$ is a perfect set (i.e., that
   every point of $M$ is an accumulation point of $M$) and, thus, $M$ is uncountable.
   More concretely, condition (ii) implies that whenever a linear piece of $f$
   ends (i.e., one of the maximal subintervals mentioned in the definition),
   then no other linear piece begins at that point, but instead there is a
   sequence of linear pieces converging to that point. And the same
   applies to the converging pieces: each of them is again approached by
   a sequence of pieces.
   The Cantor function (see e.g.~\cite{Dovgoshey-et-al:Cantor-function})
   is an example of a broken linear function
   (in which case the Cantor set is the perfect set $M$).

   Our first result shows that
   on abelian surfaces with real multiplication,
   the Seshadri function is of the same baffling complexity as the Cantor function:

\begin{introtheorem}\label{introthm:linear-segments}
   Let $X$ be a principally abelian surface, whose
   endomorphism ring is isomorphic to $\Z[\sqrt e]$ for some non-square integer $e>0$.
   Then the Seshadri function of $X$ is broken linear.
\end{introtheorem}

   This result is in stark contrast to what had been observed
   so far:
   When $E$ is a general elliptic curve,
   then the restriction of the Seshadri function on $E\times E$ to any line is
   a piecewise linear function, in the usual sense that
   each piece is adjacent to another piece~\cite{Bauer-Schulz}.
   The situation in
   Theorem~\ref{introthm:linear-segments} is
   at the other extreme:
   At no point are two pieces connected to
   each other.

   The endomorphism ring of an abelian surface with
   real multiplication is an order of a quadratic number field $\Q(\sqrt d)$.
   As the integer $e$ appearing in Theorem~\ref{introthm:linear-segments}
   is not required to be square-free, the only orders not covered there
   are those of the form
   $\Z[\frac 12+\frac 12\sqrt e]$, where $e\equiv 1 \tmod 4$.
   Surfaces with these endomorphism rings
   add another level of complexity:
   We show that on certain surfaces of this type, every line bundle
   has only one submaximal curve (which then computes its Seshadri constant),
   while
   there also exist surfaces
   of this type carrying line bundles with two submaximal curves
   (see Prop.~\ref{prop:overlapping-curves-50.000} and Prop.~\ref{prop:sequence-overlapping}).
   Interestingly,
   the conclusion of Theorem~\ref{introthm:linear-segments}
   extends to
   the former surfaces (see Theorem \ref{thm:linear-segments-never-adjacent}), whereas
   on the latter surfaces
   there exist
   boundary points of linear segments which are
   accumulation points, as well as boundary points
   where linear segments meet (see Remark~\ref{rem:border-points}).

   The discussion so far has shown how complex and subtle
   the Seshadri function
   on surfaces with real multiplication is.
   Our next result states that
   globally
   it has
   more structure
   than one might expect at this point:

\begin{introtheorem}\label{introthm:cone}
   There exists a decomposition of the ample cone into infinitely
   many subcones $\mathcal{C}_k$, $k\in\Z$, such that the
   group $G$ of isometries of $\NS(X)$ that leave
   the Seshadri function on $\Amp(X)$ invariant acts transitively
   on the set of subcones.
   In particular, the values of the Seshadri function on any subcone of the subcones $\mathcal{C}_k$ completely determine the Seshadri function on the entire ample cone.

\end{introtheorem}

   There are only few known cases where one has effective computational access to the
   Seshadri constants of all line bundles on the surface
   (the self-product $E\times E$ of a general elliptic curve being an exception again).
   Our methods provide such computational access
   for the surfaces studied here.

\begin{introtheorem}\label{introthm:algorithm}
   There is an algorithm that
   computes the Seshadri constant
   of every given ample line bundle
   on
   principally polarized
   abelian surfaces with real multiplication.
\end{introtheorem}

   The algorithm enables us to efficiently compute Seshadri
   functions and thus to provide a graphical representation
   for any given endomorphism ring
   (see examples in Sect.~\ref{sec:sample-plots}).
   Also, the analysis of the method underlying the proof of
   Theorem~\ref{introthm:algorithm}
   allows us to answer the question as to which
   data in fact determine
   the Seshadri function.
   A priori, the function could depend on the individual surface
   (or rather, on its isomorphism class).
   However, the numerical data entering the
   computation ultimately stems from the endomorphism ring, and this implies:

\begin{introcorollary}\label{introcor:Endomorphism}
   Let $X$ and $Y$ be principally polarized abelian surfaces
   with real multiplication, such that
   $\End(X)\isom\End(Y)$.
   Then, in suitable linear coordinates on $\NS(X)$ and $\NS(Y)$, their Seshadri functions
   coincide.
\end{introcorollary}

   Note that the assumption that $\End(X)$ and $\End(Y)$
   be isomorphic is strictly weaker than requiring
   that $X$ and $Y$ be isomorphic. In fact, the corollary shows that
   only countably many Seshadri functions occur, while
   the surfaces vary in two-dimensional families.

   Concerning the organization of this paper, we start in
   Sect.~\ref{sec:pell-curves-abelian} by establishing crucial properties
   of Pell divisors and submaximal curves.
   We study in Sect.~\ref{sec:local} the intervals on which curves
   can be submaximal.
   Sect.~\ref{sec:abelian-surfaces-real-mult} is devoted to the proofs of Thms.~\ref{introthm:linear-segments} and \ref{introthm:algorithm}, as well as Cor.~\ref{introcor:Endomorphism}.
   In Sect.~\ref{sec:sample-plots} we study the decomposition
   of the ample cone and prove Thm.~\ref{introthm:cone}.
   Finally, in Sect.~\ref{sec:2-submaximal}
   we investigate
   in terms of $\End(X)$
   the question
   on which surfaces there are line bundles with two submaximal
   curves.

   We would like to thank Robert Lazarsfeld for valuable suggestions
   concerning the
   exposition.

   Throughout we work over the field of complex numbers.


\section{Pell divisors and submaximal curves on abelian surfaces}\label{sec:pell-curves-abelian}
   As in the introduction, we refer to \cite[Chapt.~5]{Lazarsfeld:PAG} and \cite{Bauer-et-al:primer}
   for background on Seshadri constants. Let us just fix a few
   matters of terminology here. When we speak of the
   \emph{general upper bound}, then we mean the bound
   $\eps(L,x)\le\sqrt{L^2}$, which is valid for every ample line bundle $L$ on a
   smooth projective
   surface $S$ and for every point $x\in S$. An effective divisor
   $D$ on $S$ is called \emph{submaximal} (for $L$ at $x$), if
   $L\cdot D/\mult_x D<\sqrt{L^2}$.
   If an
   irreducible curve $C\subset S$ satisfies the equation
   $L\cdot C/\mult_x C=\eps(L,x)$, then we say that
   $C$ \emph{computes} $\eps(L,x)$.
   An irreducible curve which computes $\eps(L,x)$ for
   some ample line bundle $L$ on $S$
   will be called a \emph{Seshadri curve} on $S$.

   It was shown in \cite{Bauer:sesh-alg-sf} that
   on an abelian surface of Picard number one,
   with ample generator $L$ of the Néron-Severi group,
   there is for suitable $k\ge 1$
   a divisor $D\in|2kL|$ that computes
   the Seshadri constant of $L$.
   The number $k$ and the multiplicity of $D$ at $0$
   are governed by a Pell equation.
   We will see that a suitable notion of
   \emph{Pell divisors} (in the sense of the subsequent definition)
   also play a crucial role in the present investigation.
   The results in this section
   work on all abelian surfaces and
   do not require that the surface has real
   multiplication.

\begin{definition}
   Let $A$ be an abelian surface, and let $L$ be an ample primitive
   symmetric line bundle such that $\sqrt{L^2}\notin\Z$.
   Consider the Pell equation
   \be
      \ell^2 - L^2\cdot k^2 = 1
   \ee
   and let $(\ell,k)$ be its primitive solution.
   A divisor $D\in|2kL|^+$ with $\mult_0 D\ge 2\ell$
   is called a \emph{Pell divisor} for $L$.
\end{definition}

   Here $|2kL|^+$ denotes the linear subsystem
   of even divisors in $|2kL|$, i.e, those defined by even theta functions
   (see \cite[Sect.~4.7]{BL:CAV}).
   It will be convenient to extend the notion of Pell divisors to
   non-primitive bundles, and even to $\Q$-divisors:

\begin{definition}
   Let $A$ be an abelian surface, and let $M$
   be any ample $\Q$-line bundle on
   $A$ such that $\sqrt{M^2}\notin\Q$.
   Write $M=q L$ with a primitive ample line
   bundle $L$ and $q\in\Q$.
   A \emph{Pell divisor} for $M$ is then by definition
   a Pell divisor for $L$.
\end{definition}

   It was shown in \cite[Theorem~A.1]{Bauer-Szemberg:periods-appendix} that
   Pell divisors exist for every ample line bundle $L$ with $\sqrt{L^2}\notin\Z$.
   Their
   crucial feature is that they are {submaximal} for $L$.
   By contrast,
   the existence of submaximal divisors
   is not guaranteed when $\sqrt{L^2}\in\Z$.
   However, a dimension count shows that
   for such bundles
   there exist
   divisors
   $D\in|2L|^+$ satisfying in any event
   the weak inequality
   $L\cdot D / \mult_0D\le\sqrt{L^2}$.

   We will see that
   on abelian surfaces with real multiplication
   it is
   almost never true that $\eps(L)$ is computed by a Pell divisor of
   $L$. However, it will turn out that
   $\eps(L)$ is always computed
   by a Pell divisor of \emph{some} ample bundle on $L$, whenever
   $\sqrt{L^2}$ is irrational.
   It is for this reason that Pell divisors are crucial players
   in the present investigation.

   The following statement will prove to be a valuable tool,
   as it provides
   strong restrictions on submaximal curves.
   Also, it exhibits a situation where Pell divisors are unique.

\begin{proposition}\label{prop:C-m}
   Let $A$ be an abelian surface, and let $C\subset A$ be
   an irreducible curve that is submaximal for some ample
   line bundle. Then, putting $m=\mult_0C$, one has
   \be
      C^2-m^2=-1 \midtext{or} C^2-m^2=-4
      \,.
   \ee
   Furthermore, suppose that $C$ is not an elliptic curve,
   and write $\O_A(C)=pM$ with a primitive ample bundle $M$
   and an integer $p>0$.
   Then $\sqrt{M^2}$ is irrational, and letting
   $(\ell_0,k_0)$ be the primitive solution of the Pell
   equation $\ell^2-M^2 k^2=1$, we have:
   \begin{itemize}
   \item[\rm(i)]
      If $C^2-m^2=-1$, then the divisor $2C$ is the
      only Pell divisor for $M$ and $(\ell_0,k_0)=(m,p)$.
   \item[\rm(ii)]
      If $C^2-m^2=-4$, then the curve $C$ is the
      only Pell divisor for $M$ and $(2\ell_0,2k_0)=(m,p)$.
      In this case, the origin is the only halfperiod that
      lies on $C$.
   \end{itemize}
\end{proposition}

\begin{proof}
   The first half of the
   following argument is implicit in the proof of
   \cite[Thm.~2]{Bauer-Grimm-Schmidt:integral}.
   To provide easier access,
   we briefly make it explicit here.
   As the claim on $C^2-m^2$ is
   certainly true for elliptic curves,
   we may assume that $C$ is non-elliptic, and hence that
   $\O_A(C)$ is ample.
   The assumption that $C$ is submaximal for
   some ample line bundle $L$ then implies that $C$ is submaximal
   also for $\O_A(C)$, and in fact computes $\eps(\O_A(C))$
   (see \cite[Prop. 1.2]{Bauer-Schulz}).
   Therefore $C$ must be symmetric and therefore descends
   to a $(-2)$-curve on the smooth Kummer surface of $A$
   (cf.~proof of \cite[Thm. 6.1]{Bauer:sesh-alg-sf}).
   The multiplicities $m_i=\mult_{e_i}(C)$ at the sixteen
   halfperiods $e_i$ of $A$
   therefore satisfy the equation
   \begin{equation}\label{eqn:minus-four}
      C^2 - \sum_{i=1}^{16} m_i^2 = -4
      \,.
   \end{equation}
   Putting $m=m_1$,
   one shows
   as in the proof of \cite[Thm.~1.2]{Bauer-Grimm-Schmidt:integral}
   that only the two cases
   \begin{equation}\label{eqn:differences}
      C^2 - m^2 = -1 \quad\mbox{or}\quad
      C^2 - m^2 = -4
   \end{equation}
   are possible. This proves the first statement in the proposition.

   Write now $\O_A(C)=pM$ with a primitive ample bundle $M$
   and $p>0$.
   It follows from Eqn.~\eqnref{eqn:differences}
   that $C^2$ cannot be a perfect square, and hence that
   $\sqrt{M^2}$ is irrational.
   In the first case
   of \eqnref{eqn:differences},
   the pair $(m,p)$ satisfies the Pell
   equation
   $m^2-M^2p^2=1$.
   The minimality of the solution $(\ell_0,k_0)$ implies then that
   $m\ge\ell_0$ and $p\ge k_0$.
   On the other hand,
   as $C$ computes $\eps(M)$, we have
   for every Pell divisor $P\in|2k_0M|$ of $L$,
   \be
      \frac{M\cdot C}m \le \frac{M\cdot P}{\mult_0 P}\le \frac{M\cdot P}{2\ell_0}
   \ee
   This implies
   $\frac pm \le \frac{k_0}{\ell_0}$.
   Using the fact that both pairs $(m,p)$ and $(\ell_0,k_0)$
   solve the Pell equation, we find $m\le\ell_0$, and hence
   $(m,p)=(\ell_0,k_0)$.
   So we have $P=2C$ in this case.

   In the second case
   of \eqnref{eqn:differences},
   the number $m$ is clearly even.
   But also $p$ is even in this case, because
   all multiplicities $m_i$ are
   even (since we have $(m_1,\dots,m_{16})=(m,0,\dots,0)$).
   Therefore $\O(C)$ is totally symmetric,
   and it can therefore be written as an even multiple of
   another bundle
   (see \cite[Sect. 2, Cor. 4]{BL:CAV}).
   The upshot of this argument is that the pair $(\frac m2,\frac p2)$
   satisfies the Pell equation
   $(\frac m2)^2-M^2(\frac p2)^2=1$.
   The minimality assumption implies then that
   $\frac m2\ge\ell_0$ and $\frac p2\ge k_0$.
   But $C$ must be a component of any Pell divisor $P\in |2k_0M|$ by
   \cite[Lemma~6.2]{Bauer:sesh-alg-sf}, and so $p=2k_0$ and
   $m=2\ell_0$.
   So we have
   $P=C$ in this case.
\end{proof}

   Further, we show how two curves can intersect if they are
   submaximal for the same bundle:

\begin{proposition}\label{prop:C1C2-m1m2}
   Let $A$ be an abelian surface, and let $C_1$ and $C_2$
   be two irreducible curves on $A$ that are submaximal
   for the same ample line
   bundle $L$ on $A$, i.e.,
   \be
      \q L{C_i}0<\sqrt{L^2}
   \ee
   for $i=1,2$. Then, putting $m_i=\mult_0(C_i)$, we have
   \be
      C_1\cdot C_2 = m_1 m_2
      \,,
   \ee
   i.e., the curves $C_1$ and $C_2$ meet only at the origin, and
   their tangent cones have no common components there.
\end{proposition}

\begin{proof}
   Consider the blow-up $f:Y=\Bl_0(A)\to A$, let
   $E$ be its exceptional divisor, and
   let $C_i'\subset Y$ be
   the proper
   transform of $C_i$.
   For rational numbers $t<\sqrt{L^2}$,
   the $\Q$-divisor
   \be
      B := f^*L-tE
   \ee
   is big, because $B^2=L^2-t^2>0$ and $B\cdot f^*L=L^2>0$.
   If we take $t$ strictly between
   $\max\set{\frac{L\cdot C_1}{m_1},\frac{L\cdot C_2}{m_2}}$ and $\sqrt{L^2}$,
   then we moreover have
   \be
      B\cdot C_i' = (f^*L-tE)\cdot(f^*C_i-m_iE)
      =L\cdot C_i - tm_i < 0
      \,.
   \ee
   As a consequence,
   both $C_1'$ and $C_2'$ must be contained in the negative
   part of the Zariski decomposition of $B$, and hence their
   intersection matrix is negative definite.
   This implies that
   $C_1'^2 C_2'^2 > (C_1'C_2')^2$, i.e.,
   \begin{equation}\label{eqn:C1C2-inequality}
      (C_1^2-m_1^2)(C_2^2-m_2^2) > (C_1\cdot C_2 - m_1m_2)^2
      \,.
   \end{equation}

   We know from Prop.~\ref{prop:C-m} that
   $C_i^2-m_i^2\in\set{-1,-4}$.
   Let us first consider the case
   $C_1^2-m_1^2=C_2^2-m_2^2=-1$. Then
   inequality \eqnref{eqn:C1C2-inequality} directly implies
   $C_1\cdot C_2-m_1 m_2=0$.
   Suppose next that
   $C_1^2-m_1^2=-4$ and $C_2^2-m_2^2=-1$.
   In that case
   inequality \eqnref{eqn:C1C2-inequality} tells us that
   \be
      C_1\cdot C_2 - m_1m_2 < 2
      \,,
   \ee
   since we have in any event
   $C_1\cdot C_2 - m_1m_2 \ge 0$
   because of the intersection inequality.
   Using now
   Prop.~\ref{prop:C-m},
   we see
   that $m_1$ is an even number and
   that $C_1\equiv 2B_1$ for some line bundle
   $B_1$ on $A$.
   So we obtain
   $
      B_1\cdot C_2 - \frac{m_1}2 m_2 < 1
   $
   and hence $B_1\cdot C_2=\frac{m_1}2 m_2$, which implies
   $C_1\cdot C_2 = m_1m_2 = 0$, as claimed.
   Finally,
   if
   both $C_1^2-m_1^2$ and $C_2^2-m_2^2$ equal $-4$, then
   we get
   \be
      C_1\cdot C_2 - m_1m_2 < 4
   \ee
   from inequality \eqnref{eqn:C1C2-inequality},
   and we have
   $C_1\equiv 2B_1$,
   $C_2\equiv 2B_2$. As both $m_1$ and $m_2$ are even, this
   yields
   $B_1\cdot B_2 - \frac{m_1}2 \frac{m_2}2 = 0$,
   and this implies the assertion.
\end{proof}


\section{Submaximal curves on intervals}\label{sec:local}

\paragraph{Abelian surfaces with real multiplication.}
   Let $X$ be a simple abelian surface
   with real multiplication, i.e., such that
   $\End_\Q(X)=\Q(\sqrt d)$ for some square-free integer $d\ge 2$.
   The endomorphism ring is an order in $\End_\Q(X)$ and hence
   of the form $\End(X)=\Z+f\omega\Z$, where $f\ge 1$ is an
   integer and
   \be
      \omega=
      \begin{bycases}
         \sqrt d            & \mbox{ if } d\equiv 2,3 \tmod 4 \\
         \tfrac12(1+\sqrt d) & \mbox{ if } d\equiv 1 \tmod 4\,.
      \end{bycases}
   \ee
   For our purposes an alternative distinction of
   the possible
   cases will be
   more convenient:
   \begin{itemize}\compact
   \item
      \emph{Case 1:} $\End(X)=\Z[\sqrt{e}]$, with a non-square integer $e>0$.
   \item
      \emph{Case 2:} $\End(X)=\Z[\frac12+\frac12\sqrt{e}]$ with a non-square integer $e>0$ such that $e\equiv 1\tmod 4$.
   \end{itemize}

   If $X$ carries a principal polarization $L_0$, then we have an isomorphism
   $\varphi:\NS(X)\to\Endsym(X)=\End(X)$. It provides
   us with a lattice basis of $\NS(X)$, given by
   $L_0=\varphi\inverse(1)$ and $L_{\infty}:=\varphi\inverse(\sqrt{e})$ (resp.\ $\varphi\inverse(\tfrac12+\tfrac12\sqrt{e})$).
   The intersection matrix of this basis is
   \begin{equation}\label{eq:intersection-matrices}
      \matr{2 & 0 \\
      0       & -2e}
      \mbox{ in Case 1, and}
      \matr{2 & 1 \\
      1       & \tfrac{1-e}{2}}
      \mbox{ in Case 2.}
   \end{equation}
   This follows by considering the characteristic polynomials
   of $\sqrt{e}$ and $\tfrac12+\tfrac12\sqrt{e}$ in $\Q(\sqrt e)$
   (which coincides with the analytic characteristic polynomial
   of the endomorphism)
   and applying
   \cite[Prop.~5.2.3]{BL:CAV}.

Using the Nakai-Moishezon criterion (in the version of \cite[Cor.~4.3.3]{BL:CAV}),
and the fact that $X$ does not contain any elliptic curves,
we find:

\begin{lemma}\label{lem:ample-criterion}
Let $L$ be a line bundle on $X$ with numerical class given by $L=aL_0+bL_\infty$ for $a,b\in\Z$.
If $\End(X)=\Z[\sqrt{e}]$, then $L$ is ample if and only if
$$
      a>0 \midtext{and} a^2-eb^2>0\,,
$$
and if $\End(X)=\Z[\tfrac12+\tfrac12\sqrt{e}]$, then $L$ is ample if and only if
   $$
      \qquad\qquad a>0 \midtext{and} a^2+ab+\tfrac{1-e}{4}\,b^2>0\,.
   $$
   In either case, $L$ is ample if and only if
   $|L|\ne\emptyset$.
\end{lemma}

   From now on we will assume that
   $X$ is a principally polarized abelian surface with real multiplication.
   We are interested in its \emph{Seshadri function}
   \be
      \varepsilon :\Nef(X)\to \R,\qquad L\mapsto \varepsilon(L)=\varepsilon(L,0)
      \,.
   \ee
   Thanks to homogeneity, it
   is enough to consider this
   function on a compact cross-section of the nef cone.
   Any non-trivial nef class $L\in\NS_\R(X)$
   is a positive multiple of a class of the form
$L_t:=L_0+tL_\infty$ with suitable $t\in\R$.
Applying Lemma~\ref{lem:ample-criterion}, we see that if
$\End(X)=\Z[\sqrt{e}]$, then the line bundle $L_t$ is nef if and
only if $|t|\leq \frac{1}{\sqrt{e}}$, and if $\End(X)=\Z[\tfrac12
+ \tfrac12 \sqrt{e}]$, then $L_t$ is nef if and only if
$-\tfrac{2}{\sqrt{e}+1}\leq t\leq \tfrac{2}{\sqrt{e}-1}$.
Ampleness holds when the inequalities are strict.
We denote by $\mathcal N(X)=[-\frac{1}{\sqrt{e}},\frac{1}{\sqrt{e}}]$ and $\mathcal N(X)=[-\tfrac{2}{\sqrt{e}+1}, \tfrac{2}{\sqrt{e}-1}]$, respectively, the interval where $L_t$ is nef.
This interval $\mathcal N(X)$ is a model for the cross-section of
the nef cone, and therefore we will also write $L_t\in \mathcal N(X)$
instead of $t\in \mathcal N(X)$.
So for every nef $\R$-line bundle, the ray $\R_{> 0}L$ has a unique representative in $\mathcal N(X)$, and we
may consider the Seshadri function as
\be
\varepsilon:\mathcal N(X)\to\R,\qquad t\mapsto \varepsilon(L_t)\,.
\ee

Any effective divisor $D$ defines a linear function
\be
\ell_D:\R\to \R,\qquad t\mapsto \frac{D\cdot L_t}{\mult_0 D}
\ee
which computes the Seshadri quotient of the divisor $D$ for any line bundle $L_t$.
We denote the open subset containing all ample line bundles $L_t$, whose Seshadri quotient with $D$ is submaximal, by $I_D$, i.e.,
\be
I_D:= \left\{ t\in \mathcal N(X)\,\,\vrule\,\, \ell_D(t)<\sqrt{L_t^2}\right\}\,,
\ee
and we
call $I_D$ the \textit{submaximality interval} of $D$.

It is a result of Szemberg \cite[Prop.~1.8]{Szemberg} that on any smooth projective surface $S$ an ample line bundle can have at most $\rho(S)$ (two, in our case) submaximal curves at any given point.
Using the restrictions derived from Prop.~\ref{prop:C-m} and \ref{prop:C1C2-m1m2} we show that in many cases only \emph{one} curve can exist:

\begin{theorem}\label{thm:one-submax-curve-lb}
Let $L$ be any ample line bundle on $X$ with $\varepsilon(L)<\sqrt{L^2}$.
Suppose that either
\begin{itemize}\compact
\item $\End(X)=\Z[\sqrt{e}]$ for a non-square integer $e>0$, or
\item $\End(X)=\Z[\tfrac12 + \tfrac12 \sqrt{e}]$ for a non-square integer $e>0$, such that $e\equiv 1$ modulo $4$ and $e$ has a prime factor $p$ with $p\equiv 5$ or $7$ modulo $8$,
\end{itemize}
holds.
Then there exists exactly one irreducible curve $C$ that is submaximal for $L$.
\end{theorem}

\begin{proof}
We will show that the restrictions given in Prop.~\ref{prop:C-m} and Prop.~\ref{prop:C1C2-m1m2} cannot hold for two submaximal curves.
In fact, we will show that the following equations can never be satisfied by \textit{any} two ample line bundles $L_1$ and $L_2$ and two positive integers $m_1$ and $m_2$:
\begin{enumerate}\compact

\item[(i)] $L_1^2=m_1^2-1$ for $m_1>1$\,,
\item[(ii)] $L_2^2=m_2^2-1$ for $m_2>1$\,,
\item[(iii)] $L_1\cdot L_2=m_1m_2$\,.
\end{enumerate}
Note that this also includes the case $C^2=m^2-4$ from Prop.~\ref{prop:C-m}, because in this case $\O_X(C)$ is an even multiple of another line bundle and dividing the equation by $4$ leads to an equation of the form (i).

First we treat the more immediate case: $\End(X)=\Z[\sqrt{e}]$, where $e$ is a non-square positive integer.
Assume that there exist ample line bundles $L_1$ and $L_2$ satisfying (i) and (ii), with their numerical classes given by $L_i\equiv a_iL_0+b_iL_\infty\,$ for $i=1,2$.
Then $m_1$ and $m_2$ must be odd, since $L_i^2$ is even.
But the intersection number for any two line bundles on $X$ is even,
\be
(a_1L_0 + b_1L_\infty)\cdot(a_2L_0+b_2L_\infty)=2a_1a_2 -2eb_1b_2
\,,
\ee
and hence it can never equal $m_1m_2$.

Next we treat the more subtle case: $\End(X)=\Z[\tfrac12 + \tfrac12 \sqrt{e}]$, where $e$ is a non-square positive integer with $e\equiv 1$ modulo $4$, which has a prime factor $p$ with $p\equiv 5$ or $7$ modulo $8$.
The crucial idea in this case is to consider the three equations modulo $p$.
Assume that $L_1$ and $L_2$ are two line bundles satisfying (i)--(iii), with their numerical classes given by $L_i\equiv a_iL_0+b_iL_\infty\,$ for $i=1,2$.
If we consider the equations
\begin{enumerate}\compact
\item[(i)]   $2L_1^2=4a_1^2+4a_1b_1 +(1-e)b_1^2  = 2m_1^2-2$\,,
\item[(ii)]  $2L_2^2=4a_2^2+4a_2b_2 +(1-e)b_2^2  = 2m_2^2-2$\,,
\item[(iii)] $2L_1\cdot L_2 = 4a_1a_2 + 2a_1b_2 + 2a_2b_1 + (1-e)b_1b_2 = 2m_1m_2$\,,
\end{enumerate}
modulo $p$ and replace $2a_i+b_i$ by $c_i$ for $i=1,2$, then the equations can be expressed by bilinear forms over the finite field $\mathbb{F}_{p}$.
For (i) and (ii) we obtain
\begin{enumerate}
\item[(I)] $\qquad 	\left(
		\begin{array}{c}
            c_1 \\
            m_1
		\end{array}
      	\right)^T
      	\left(
         \begin{array}{cc}
            1 & 0  \\
            0 & -2
         \end{array}
		\right)
      	\left(
		\begin{array}{c}
            c_1    \\
            m_1
		\end{array}
		\right)=c_1^2-2m_1^2=-2$\,,
\item[(II)] $\qquad 	\left(
		\begin{array}{c}
            c_2 \\
            m_2
		\end{array}
      	\right)^T
      	\left(
         \begin{array}{cc}
            1 & 0  \\
            0 & -2
         \end{array}
		\right)
      	\left(
		\begin{array}{c}
            c_2    \\
            m_2
		\end{array}
		\right)=c_2^2-2m_2^2=-2$\,.
\end{enumerate}
It follows that $(c_i,m_i)\neq (0,0)\in \mathbb{F}_{p}^2$.
For equation (iii) we find that
\begin{enumerate}
\item[(III)] $\qquad
		\left(
		\begin{array}{c}
            c_1 \\
            m_1
		\end{array}
      	\right)^T
      	\left(
         \begin{array}{cc}
            1 & 0  \\
            0 & -2
         \end{array}
		\right)
      	\left(
		\begin{array}{c}
            c_2    \\
            m_2
		\end{array}
		\right)=\left(
		\begin{array}{c}
            c_1 \\
            -2m_1
		\end{array}
      	\right)^T
      	\cdot
      	\left(
		\begin{array}{c}
            c_2    \\
            m_2
		\end{array}
		\right)=0$
\end{enumerate}
and, therefore, we obtain
\be
		\left(
		\begin{array}{c}
            c_2    \\
            m_2
		\end{array}
		\right)\in \ker
		\left(
		\begin{array}{c}
            c_1 \\
            -2m_1
		\end{array}
      	\right)^T
		=\,
		\left\{\lambda\left(
		\begin{array}{c}
            -2m_1    \\
            c_1
		\end{array}
		\right)\,\vrule\, \lambda\in \mathbb F_p\right\}\,,
\ee
i.e., $c_2=-2m_1\lambda$ and $m_2=c_1\lambda$ for some $\lambda\in \mathbb{F}_{p}$.
Using (I) and (II) we obtain
\be
		-2=\left(
		\begin{array}{c}
            c_2 \\
            m_2
		\end{array}
      	\right)^T
      	\left(
         \begin{array}{cc}
            1 & 0  \\
            0 & -2
         \end{array}
		\right)
      	\left(
		\begin{array}{c}
            c_2    \\
            m_2
		\end{array}
		\right)=
		\lambda^2
		(4m_1^2-2c_1^2)
		=4\lambda^2\,.
\ee
This implies that $-2$ is a quadratic residue modulo $p$.
But as $p\equiv 5$ or $7$ modulo $8$,
this is impossible, and thus we arrive at a contradiction.
\end{proof}

As a consequence of Thm.~\ref{thm:one-submax-curve-lb} we observe:

\begin{corollary}\label{cor:linear-case}
Let $\End(X)$ be as in Thm.~\ref{thm:one-submax-curve-lb}.
Then for every ample $\R$-line bundle $L_\lambda$ with $\varepsilon(L_\lambda)<\sqrt{L_\lambda^2}$ the Seshadri function is given by a linear function in a neighborhood of $L_\lambda$.
\end{corollary}

\begin{proof}
Let $L_\lambda$ be an ample $\R$-line bundle with $\varepsilon(L_\lambda)<\sqrt{L_\lambda^2}$ and let $C$ be any Seshadri curve of $L_\lambda$.
By the previous Thm.~\ref{thm:one-submax-curve-lb} the curve $C$ is the only submaximal curve for every $\Q$-line bundle in $I_C$.
Assume now that there exists an ample $\R$-line bundle $L_t\in I_C$ with two submaximal curves.
By continuity both curves remain submaximal in a neighborhood of $L_t$ and, thus, there also exist $\Q$-line bundles which also have two submaximal curves.
This, however, is impossible by Thm.~\ref{thm:one-submax-curve-lb}.
\end{proof}

We will see that the assumption $\varepsilon(L_\lambda)<\sqrt{L_\lambda^2}$ is essential for the validity of the statement in the corollary,
and in fact we will show that the local behavior in the remaining case $\varepsilon(L_\lambda)=\sqrt{L_\lambda^2}$ is surprisingly intricate (see Cor.~\ref{cor:upper-bound-case}).

Computer-assisted calculations suggest that Thm.~\ref{thm:one-submax-curve-lb} is in fact an \engqq{if and only if} statement, which means that in the remaining cases there should always exist a line bundle with two submaximal curves.
In Sect.~\ref{sec:2-submaximal} we will show how the existence of a line bundle with two submaximal curves can be verified using computer-assisted calculations.
Furthermore, we will provide a sequence of numbers $e_n$ with the property that there exists a line bundle with two submaximal curves on any abelian surface with $\End(X)=\Z[\tfrac12+\tfrac12\sqrt{e_n}]$. \\

Before we continue studying the local behavior of the Seshadri function in the case where line bundles can have two submaximal curves, we prove a useful relation between submaximality intervals and reducibility of effective divisors:

\begin{lemma}\label{lem:reducible-criterion}
Let $D$ be an effective divisor on $X$ which is submaximal for some ample line bundle.
If there exists another effective divisor $D'$ whose submaximality interval $I_{D'}$ satisfies $I_D\subsetneq I_{D'}$, then $D$ is reducible.
\end{lemma}

\begin{proof}
Let $I_{D}=(a,b)$ and $I_{D'}=(c,d)$ be the submaximality intervals of $D$ and $D'$, respectively.
Denoting by $F$ the general upper bound function $t\mapsto \sqrt{L_t^2}$,
the linear function $\ell_D$ is given by the straight line joining the points $(a,F(a))$ and $(b,F(b))$ and, respectively, $\ell_{D'}$ by joining the points $(c,F(c))$ and $(d,F(d))$.
Since $F$ is strictly concave, the linear function $\ell_{D'}$ is strictly smaller than $\ell_{D}$ in $I_{D'}$.
Therefore $D$ can never compute the Seshadri constant for any line bundle.
However, if $D$ were irreducible, then $D$ would compute its own Seshadri constant (see \cite[Prop. 1.2]{Bauer-Schulz}), which is a contradiction.
\end{proof}

By \cite[Prop.~1.8]{Szemberg}, the number of curves that can be submaximal for an individual ample line bundle $L_t$ is bounded.
We will now show that the number remains bounded even when all line bundles in an open neighborhood of $L_t$ are considered, provided that
$\varepsilon(L_t)<\sqrt{L_t^2}$. This is a consequence of the following lemma.

\begin{lemma}\label{lem:4-curves-submaxinterval}
Let $D$ be an effective divisor on $X$ which is submaximal for an ample line bundle $L_t$.
Then there exists at most four irreducible curves which are submaximal for some line bundles in $I_D$.

Moreover, if $D$ is irreducible, then there exists at most three irreducible curves which are submaximal for some line bundles in $I_D$.
\end{lemma}

\begin{proof}
Assume, there exists five pairwise distinct irreducible curves $C_1,...,C_5$ which are submaximal for some line bundles in $I_D=(a,b)$.
Let $I_{C_i}=(a_i,b_i)$ be the submaximality interval of $C_i$ and let $L_{t_i}\in \mathcal N(X)$ be the unique representative of $\O_X(C_i)$.
We will show that the submaximality interval of $C_3$ is contained in $I_D$, which by Lemma~\ref{lem:reducible-criterion} would imply that $C_3$ is reducible.

Since $C_i$ is submaximal for some ample line bundle, $C_i$ is submaximal for $\O_X(C_i)$ by \cite[Prop. 1.2]{Bauer-Schulz}.
Therefore, $C_i$ is submaximal for $L_{t_i}$ and, thus, $t_i\in I_{C_i}$.
Moreover, since $C_i$ is the only submaximal curve for $\O_X(C_i)$, we have $t_i\notin I_{C_j}$ for $i\neq j$.
By assuming $t_1<t_2<t_3<t_4<t_5$ we deduce for $i=2,3,4$ that
$$
(a_i,b_i)\subset (t_{i-1},t_{i+1}) \qquad \mbox{ and }\qquad t_i\in (b_{i-1},a_{i+1})\,. \eqno(\ast)
$$
The submaximality intervals $(a_1,b_1)$ and $(a_5,b_5)$ have to intersect with $(a,b)$, because by assumption $C_1$ and $C_5$ are submaximal for some line bundles in $I_D$ and, thus, we have $a<b_1$ and $a_5<b$.
Furthermore, $(\ast)$ implies that $t_2,t_3,t_4\in (b_1,a_5)$ and, as a consequence, the interval $(t_{2},t_{4})$ is contained in $(b_1,a_5)$ and, therefore, in $I_D$.
Since $(a_3,b_3)$ is contained in $(t_{2},t_{4})$, it is also contained in $I_D$.
This, however, implies that $C_3$ is reducible by Lemma~\ref{lem:reducible-criterion}, which is a contradiction.

For the second statement, we assume there exists three irreducible curves.
Using the same notation and arguments as above, it follows that $t_2\in (b_1,a_3)$, $a<b_1$, and $a_3<b$.
Hence, we have $t_2\in  (b_1,a_3)\subset I_D$.
But this means that $L_{t_2}$ has $C_2$ and $D$ as submaximal curves, which is a contradiction.
\end{proof}

Hence, we conclude for the local structure of the Seshadri function:

\begin{corollary}
For every ample $\R$-line bundle $L_t$ with $\varepsilon(L_t)<\sqrt{L_t^2}$ the Seshadri function is locally a piecewise linear function, i.e., it is locally the minimum of at most two linear functions.
\end{corollary}
As before, the assumption $\varepsilon(L_t)<\sqrt{L_t^2}$ is essential for this statement to be true (see Remark~\ref{rem:border-points}).

Clearly, if a line bundle $L$ has two submaximal curves, then there exists a neighborhood of $L$ such that every line bundle has two submaximal curves, since any submaximal curve will remain submaximal in a neighborhood of $L$.
On the other hand, we show that every submaximal curve gives rise to an open interval, in which it is the only submaximal curve:
\begin{proposition}\label{prop:sesh-curve-1-submaximal}
Let $C\equiv qL_0+pL_\infty$ be an irreducible curve that is submaximal for some ample line bundle $L$ on $X$.
Then there exists a neighborhood $U$ of $L_{\frac{p}{q}}$ in $\mathcal N(X)$ such that $C$ is the only submaximal curve for all line bundles in $U$.
In particular, the Seshadri function coincides with $\ell_C$ in $U$.
\end{proposition}

\begin{proof}
Since $C$ is submaximal for some ample line bundle $L$, we know that $C$ is also submaximal for $\O_X(C)$ by \cite[Prop. 1.2]{Bauer-Schulz}, and
in fact $C$ is the only submaximal curve for $\O_X(C)$, since every $\O_X(C)$-submaximal curve has to be a component of $C$ by \cite[Lemma~5.2]{Bauer:sesh-alg-sf}.
Thus, $C$ is the only submaximal curve for $L_{\frac{p}{q}}$.
Applying Lemma.~\ref{lem:4-curves-submaxinterval}, there exist at most two other curves, which are submaximal for some line bundle $L'\in I_C$.
Thus, the only possibility in which no such neighborhood of $L_\frac{p}{q}$ exists, is the case where one of the other curves $C'$ satisfies $\O_X(C)\cdot C'/\mult_0(C')=\sqrt{C^2}$.
This, however, implies that $C'$ is a component of $C$ by \cite[Lemma~5.2]{Bauer:sesh-alg-sf}.
\end{proof}

\section{Seshadri function on abelian surfaces with real multiplication}\label{sec:abelian-surfaces-real-mult}

In this section we will develop a method to algorithmically compute the Seshadri constant for any ample $\Q$-line bundle on $X$, proving Theorem~\ref{introthm:algorithm} stated in the introduction.
Furthermore, we will see that the local structure of the Seshadri function has unexpected behavior at $L_\lambda$ if $\varepsilon(L_\lambda)=\sqrt{L_\lambda^2}$.
Our strategy is to make use of Pell divisors in such a way that it is not necessary to explicitly know their multiplicity, but to use their \textit{expected multiplicity} given by the Pell solution.

\begin{definition}\label{def:pell-bound}
Let $L_\lambda$ be an ample $\Q$-line bundle with $\sqrt{L_\lambda^2}\notin\Q$ and let $q\in \N$ be the unique integer such that $qL_\lambda$ is a primitive $\Z$-line bundle, i.e., $q$ is the denominator of a coprime representation of $\lambda=\frac{p}{q}$.
Denote by $(l,k)$ the primitive solution of the Pell equation $x^2-(qL_\lambda)^2 y^2=1$.
We call
\be
\pi_\lambda:\R\to\R,\qquad t\mapsto \frac{kqL_\lambda\cdot L_t}{l}
\ee
the \textit{Pell bound} at $L_\lambda$, and \be
J_\lambda=\{t\in\mathcal N(X)\,\,\vrule\,\, \pi_\lambda(t)<\sqrt{L_t^2}\}
\ee
the \textit{submaximality interval} of $\pi_\lambda$.
\end{definition}
So if $\sqrt{L_\lambda^2}\notin\Q$ and $P$ is a Pell divisor of $L_\lambda$, then we have the following chain of inequalities:
\be
     \eps(L_\lambda) \le \frac{L_\lambda\cdot P}{\mult_0P}
             \le \pi_\lambda(\lambda)
             < \sqrt{L_\lambda^2}\,.
\ee
Moreover,
$\pi_\lambda$ is an upper bound for the Seshadri function in the submaximality interval~$J_\lambda$:
\be
   \eps(L_t)\le \frac{L_t\cdot P}{\mult_0P} \le \pi_\lambda(t)              < \sqrt{L_\lambda^2} \righttext{for all $t\in J_\lambda$.}
\ee

We will now establish two important connections between submaximal curves and Pell bounds.
First we prove that every submaximal curve has a unique representative in the set of Pell bounds.
Secondly, we will exhibit a relation between the submaximality interval of a Seshadri curve $C$ of $L_\lambda$ and the submaximality interval of the Pell bound $\pi_\lambda$.

\begin{proposition}\label{prop:irred-pell-bound}
Let $L_\lambda$ be an ample $\Q$-line bundle with $\sqrt{L_\lambda^2}\notin\Q$ and let $C\equiv qL_0+pL_\infty$ be an irreducible curve that is submaximal for some ample line bundle $L$ on $X$.
Then the following are equivalent:
\begin{enumerate}\compact
\item[\rm(i)] Either $C$ or $2C$ is the unique Pell divisor of $L_\lambda$.
\item[\rm(ii)] The linear functions $\ell_C$ and $\pi_\lambda$ coincide.
\item[\rm(iii)] We have $\lambda=\frac pq$.
\end{enumerate}
\end{proposition}

\begin{proof}
The equivalence of (i) and (iii) is an immediate consequence of Prop.~\ref{prop:C-m}.
Furthermore, the implication (i) $\Rightarrow$ (ii) also follows from Prop.~\ref{prop:C-m}, since it shows that the multiplicity of $C$ coincides with the expected multiplicity given by the Pell solution,
and therefore the linear functions $\ell_C$ and $\pi_\frac{p}{q}$ coincide.

For the implication (ii) $\Rightarrow$ (iii) we have to show that $\ell_C=\pi_\lambda$ implies $\lambda=\frac{p}{q}$.
By Prop.~\ref{prop:sesh-curve-1-submaximal} the linear function $\ell_C$ coincides with the Seshadri function in an open neighborhood $U$ of $L_\frac{p}{q}$.
For any Pell divisor $P$ of $L_\lambda$ we have
\be
\varepsilon(L_t)=\ell_C(t)\leq \ell_P(t)\leq \pi_\lambda(t)\qquad \mbox{ for all } t\in U\,,
\ee
and hence the linear function $\ell_P$ coincides with $\ell_C$, since by assumption $\pi_\lambda=\ell_C$.

We claim that for every component $C'$ of $P$ the linear functions $\ell_{C'}$ and $\ell_C$ also coincide.
For this, assume that there exists a $t_0\in U$ such that $\ell_C(t_0)<\ell_{C'}(t_0)$.
Then, upon writing $P=C'+R$, we have
\be
\frac{C\cdot L_{t_0}}{\mult_0 C}=\ell_C({t_0})=\ell_P({t_0})=\frac{(C'+R)\cdot L_{t_0}}{\mult_0 C' + \mult_0 R}\,.
\ee
This, however, implies that
\be
\frac{C\cdot L_{t_0}}{\mult_0 C}>\frac{R\cdot L_{t_0}}{\mult_0 R}\,,
\ee
which is impossible, since $C$ computes the Seshadri constant $\varepsilon(L_{t_0})$.

So we have shown that $C$ and any component $C'$ of the Pell divisor $P$ define the same linear function.
This means, in particular, that any component $C'$ of $P$ is also submaximal for the line bundle $\O_X(C)$.
But $\O_X(C)$ has only $C$ as a submaximal curve by \cite[Lem.~5.2]{Bauer:sesh-alg-sf}, and, therefore, $P=kC$ for $k\in \N$.
This implies that $L_\frac{p}{q}$ and $L_\lambda$ are rational multiples of each other.
But in $\mathcal N(X)$ this is only possible if $\lambda=\frac{p}{q}$.
\end{proof}

\begin{proposition}\label{prop:seshadri-submaximal-interval}
Let $L_\lambda$ be an ample $\Q$-line bundle with $\sqrt{L_\lambda^2}\notin\Q$ and let $J_\lambda=(t_1,t_2)$ be the submaximality interval of the Pell bound $\pi_\lambda$.
Then every Seshadri curve $C$ of $L_\lambda$ is submaximal on $(t_1,\lambda)$ or on $(\lambda,t_2)$.
\end{proposition}
\begin{proof}
Assume that $C$ is not submaximal on $(\lambda,t_2)$, i.e., $\ell_C(t_2)>\sqrt{L_{t_2}^2}=\pi_\lambda(t_2)$.
Furthermore, since $C$ is a Seshadri curve of $L_\lambda$, we have $\ell_C(\lambda)\leq \pi_\lambda(\lambda)$.
Therefore
the slopes $m_\lambda$ of $\pi_\lambda$ and $m_C$ of $\ell_C$ satisfy $m_\lambda < m_C$.
But this implies that $\ell_C(t_1)<\pi_\lambda(t_1)$
and therefore $C$ is submaximal on $(t_1,\lambda)$.
\end{proof}

We will need the
submaximality intervals of Pell bounds in the following explicit form:

\begin{lemma}\label{lem:interval-borders}
Let $L_\lambda$ be an ample $\Q$-line bundle with $\sqrt{L_\lambda^2}\notin\Q$, and let $l,k$ and $q$ be as in Def.~\ref{def:pell-bound}.
If $\End(X)=\Z[\sqrt{e}]$, then the submaximality interval $J_\lambda$ of $\pi_\lambda$ is given by
\be
J_\lambda=\left(\frac{2ek^2q^2\lambda - l \sqrt{e}}{e(2k^2q^2+1)}, \frac{2ek^2q^2\lambda + l \sqrt{e}}{e(2k^2q^2+1)} \right)\,,
\ee
and, if $\End(X)=\Z[\tfrac12 + \frac12 \sqrt{e}]$, then the submaximality interval $J_\lambda$ of $\pi_\lambda$ is given by
\be
J_\lambda=\left(\frac{2+2ek^2q^2\lambda - 2l \sqrt{e}}{(e-1)+2eq^2k^2},
\frac{2+2ek^2q^2\lambda + 2l \sqrt{e}}{(e-1)+2eq^2k^2}
\right)\,.
\ee
 \end{lemma}
\begin{proof}
The interval limits of the submaximality interval $J_\lambda=(t_1,t_2)$ are the solutions $t$ of the equation
\be
\sqrt{L_t^2}=\pi_\lambda(t)\,.
\ee
In the case $\End(X)=\Z[\sqrt{e}]$, the solutions are
\be
t_{1,2}=\frac{2ek^2q^2\lambda \mp l \sqrt{e}\sqrt{l^2-(qL_\lambda)^2k^2}}{e(2ek^2q^2\lambda^2+l^2)}\,.
\ee
Upon applying the Pell equation $l^2-(qL_\lambda)^2 k^2=1$, these solutions can be expressed by
\be
t_{1,2}=\frac{2ek^2q^2\lambda \mp l \sqrt{e}}{e(2k^2q^2+1)}\,.
\ee
The case $\End(X)=\Z[\tfrac12 + \tfrac12 \sqrt{e}]$ is computed analogously.
\end{proof}
As the linear function $\ell_C$ of a submaximal curve $C$ coincides with a Pell bound, the interval borders for Seshadri curves have the same structure.
This reveals an interesting behavior of submaximality intervals of Seshadri curves:
\begin{proposition}\label{prop:adjacent}
Let $C_1$ and $C_2$ be two submaximal curves on $X$.
Then the submaximality intervals $I_{C_1}$ and $I_{C_2}$ are never adjacent to each other, i.e., if $I_{C_1}=(t_1,t_2)$ and $I_{C_2}=(s_1,s_2)$, then $t_1\neq s_2$ and $t_2\neq s_1$.
\end{proposition}

\begin{proof}
Using the computation of the interval limits of Lemma~\ref{lem:interval-borders}, it follows that the left-hand side of the interval is always of the form $a-b\sqrt{e}$ for some $a\in \Q$ and $b\in \Q^+$, whereas the right-hand side is of the form $a'+b'\sqrt{e}$ for some $a'\in \Q$ and $b'\in \Q^+$.
Since $1$ and $\sqrt{e}$ form a basis of the $\Q$-vector space $\Q(\sqrt{e})$, they can never coincide.
\end{proof}

\begin{corollary}\label{cor:upper-bound-case}
Let $L_\lambda$ be any ample $\R$-line bundle such that $\varepsilon(L_\lambda)=\sqrt{L_\lambda^2}$.
For every neighborhood $U$ of $\lambda$, the Seshadri function is the pointwise infimum of infinitely many linear functions $\pi_\mu$, but it is not a piecewise linear function on $U$.
\end{corollary}

\begin{proof}
In any neighborhood $U$ of $\lambda$, the rational numbers $\mu\in U$ with $\varepsilon(L_\mu)<\sqrt{L_\mu^2}$ are dense in $U$.
Thus, by continuity of the Seshadri function we can express the Seshadri constant for any value $t\in U$ as an infimum of linear functions $\pi_\mu$.
The only possibility for the Seshadri function to be a piecewise linear function in a neighborhood of $\lambda$ is, if the Seshadri function is computed near $\lambda$ by two linear functions $\ell_1$ and $\ell_2$ with $\ell_1(\lambda)=\ell_2(\lambda)=\sqrt{L_\lambda^2}$.
This, however, is impossible by Prop.~\ref{prop:adjacent}.
\end{proof}

By combining Cor.~\ref{cor:linear-case} and \ref{cor:upper-bound-case} we deduce the following more general version of Theorem~\ref{introthm:linear-segments} stated in the introduction.

\begin{theorem}\label{thm:linear-segments-never-adjacent}
Let $\End(X)$ be as in Thm.~\ref{thm:one-submax-curve-lb}.
Then the Seshadri function is broken linear.
\end{theorem}
\begin{proof}
It follows from Cor.~\ref{cor:linear-case} that for every point $t\in\mathcal N(X)$ with $\varepsilon(L_t)<\sqrt{L_t^2}$ the Seshadri function is a linear function in a neighborhood of $t$.
By Prop.~\ref{prop:adjacent} the maximal intervals, on which the Seshadri function is linear are never adjacent to each other.
Lastly, we have to argue that the set $M(X)=\set{t\in\mathcal N(X)\,\,\vrule\,\, \varepsilon(L_t)=\sqrt{L_t^2}}$ is nowhere dense and non empty.
For this we consider ample line bundles of the form $L=qL_0+4qL_\infty$ for odd $q\in\N$ and $p\in \Z$.
In this cases $L^2$ can never be a square number as $L^2\equiv 2\tmod 4$.
This yields a dense subset of lines bundles $L_{4q/p}$ in $\mathcal N(X)$ with $\varepsilon(L_{4p/q})<\sqrt{L_{4p/q}^2}$.
As the Seshadri function is continuous, we get for each line bundle $L_{4q/p}$ an open neighborhood on which the Seshadri function is submaximal.
Thus, $M(X)$ is a nowhere dense subset of $\mathcal N(X)$.
Explicit computations show that the Seshadri curve $C\in |4L_0|$ of $L_0$ with $\mult_0 C=6$ is not submaximal on $\mathcal N(X)$ and, therefore, the interval borders of the submaximality interval $I_C$ are contained in $M(X)$.
\end{proof}

%

\begin{remark}\label{rem:border-points}
Suppose that on $X$ there is a line bundle $L_\lambda$ with two submaximal curves.
Then there exists a neighborhood of $L_\lambda$ in which every line bundle has two submaximal curves, and thus there exist linear segments of the Seshadri function that are adjacent to each other.
On the other hand, we have seen in Prop.~\ref{prop:sesh-curve-1-submaximal} that there are also neighborhoods, in which only one submaximal curve exists.
Furthermore, using Prop.~\ref{prop:C1C2-m1m2} one can show that every line bundle in the submaximality interval $I_0$ of $L_0$ has only one submaximal curve $C$, which is the unique Pell divisor of $L_0$.
Consequently, the limit points of this submaximal interval are accumulation points of (piecewise) linear segments.
So in this case, as in the situation of Thm.~\ref{thm:linear-segments-never-adjacent}, the Seshadri function
does not consist of only finitely many linear pieces.
\end{remark}

We return to the submaximality interval $J_\lambda=(t_1,t_2)$ of a Pell bound $\pi_\lambda$ by
providing an upper bound for its length:

\begin{lemma}\label{lem:interval-length}
Let $L_\lambda$ be an ample $\Q$-line bundle with $\sqrt{L_\lambda^2}\notin\Q$, and let $l,k$ and $q$ be as in Def.~\ref{def:pell-bound}.
Then the interval length of $J_\lambda=(t_1,t_2)$ is bounded by
\be
t_2-t_1 <\frac{\sqrt{11}}{q\sqrt{e}}\,.
\ee
\end{lemma}
\begin{proof}
Using the fact that the Pell equation is equivalent to
\be
\frac{l}{kq}=\sqrt{L_\lambda^2+\frac{1}{k^2q^2}}\,,
\ee
the interval length can be determined via
Lemma~\ref{lem:interval-borders}: In the
case $\End(X)=\Z[\sqrt{e}]$ we get
$$
t_2 - t_1 = \frac{2\tfrac{l}{kq}}{\sqrt{e}(2kq+\tfrac{1}{kq})}
=\frac{2\sqrt{L_\lambda^2+\tfrac{1}{k^2q^2}}}{\sqrt{e}(2kq+\tfrac{1}{kq})}
=\frac{2\sqrt{2-\frac{2ep^2k^2-1}{k^2q^2}}}{\sqrt{e}(2kq+\tfrac{1}{kq})}
< \frac{\sqrt{2}}{q\sqrt{e}}\,,
$$
and in the case $\End(X)=\Z[\tfrac12 + \tfrac12 \sqrt{e}]$ we obtain
$$
t_2 - t_1 = \frac{4\frac{l}{kq} \sqrt{e}}{\frac{e-1}{kq}+2ekq}=\frac{4\sqrt{e}\sqrt{2+2\lambda-\frac{e-1}{2}\lambda^2+\frac{1}{k^2q^2}} }{\frac{e-1}{kq}+2ekq}< \frac{2\sqrt{2+\frac{2}{e-1}+\frac{1}{k^2q^2}} }{q\sqrt{e}}<\frac{\sqrt{11}}{q\sqrt{e}}\,.
$$
In the third step we replaced the term $2\lambda - \tfrac{e-1}{2}\lambda^2$ with its maximum value $\frac{2}{e-1}$, and we use $e\geq 5$ in the last step.
Also, we made use of the inequality $\tfrac{1}{k^2q^2}\leq \tfrac{1}{4}$, which can be verified by explicitly considering all possible Pell solutions for $q=1$.
\end{proof}
\begin{remark}
In the last step of the proof we could have used the more direct estimate $\tfrac{1}{k^2q^2}\leq 1$, which implies
\be
\frac{2\sqrt{2+\frac{2}{e-1}+\frac{1}{k^2q^2}} }{q\sqrt{e}}<\frac{\sqrt{14}}{q\sqrt{e}}\,
\ee
However, it turns out that this upper bound is not sufficient for our purposes (in particular, for the proof of Prop.~\ref{prop:sequence-overlapping}).
\end{remark}

The previous Lemma yields the following:

\begin{corollary}\label{cor:interval-finite-pell-bounds}
For any given interval $I\subset \mathcal N(X)$ there exist only finitely many Pell bounds $\pi_\lambda$ that are submaximal on $I$.
\end{corollary}

\begin{proof}
Let $s$ be the length of $I$ and let $\pi_\lambda$ be a Pell bound that is submaximal on $I$.
Then $s$ is at most the length of $J_\lambda$, so that it
follows from Lemma~\ref{lem:interval-length} that the denominator of $\lambda=\frac{p}{q}$ satisfies
\be
s \leq \frac{\sqrt{11}}{q\sqrt{e}}\,.
\ee
Thus, $\lambda$ has to be contained in the finite set
\be
\set{\frac{a}{b} \in \mathcal N(X) \with 1\leq b\leq \frac{\sqrt{11}}{s\sqrt{e}}\,, \, \gcd(a,b)=1 }\,.
\ee
\end{proof}

By combining Prop.~\ref{prop:seshadri-submaximal-interval} and Cor.~\ref{cor:interval-finite-pell-bounds} we will obtain a purely numerical method to compute the Seshadri constant and determine the
Seshadri curves of $L_\lambda$:

\begin{proposition}\label{prop:crit-pell-irred2}
Let $L_\lambda$ be an ample $\Q$-line bundle with $\sqrt{L_\lambda^2}\notin\Q$, and $J_\lambda=(t_1,t_2)$ be the submaximality interval of $\pi_\lambda$.
Let $s(\lambda)=\min\{\lambda-t_1,t_2-\lambda\}$, and consider the finite set $A_\lambda:=\set{\frac{a}{b} \in \mathcal N(X) \with 1\leq b\leq \frac{\sqrt{11}}{s(\lambda)\sqrt{e}}\,, \, \gcd(a,b)=1  }$.
Then the Seshadri constant of $L_\lambda$ is given by
\be
\varepsilon(L_\lambda)=\min\{\pi_\mu(\lambda) \,\vrule\, \mu\in A_\lambda\}\,.
\ee
Moreover, every Seshadri curve $C$ of $L_\lambda$ is represented by a unique Pell bound $\pi_\tau$ with $\tau\in A_\lambda$ and $\varepsilon(L_\lambda)=\pi_\tau(\lambda)$.
\end{proposition}

\begin{proof}
By Prop.~\ref{prop:seshadri-submaximal-interval} any Seshadri curve $C$ of $L_\lambda$ is submaximal either on $(t_1,\lambda)$ or on $(\lambda, t_2)$, and therefore it is submaximal on an interval of length $s=\min\{\lambda-t_1,t_2-\lambda\}$.
By Prop.~\ref{prop:irred-pell-bound} the linear function $\ell_C$ coincides with a unique Pell bound $\pi_\tau$, and therefore their submaximality intervals coincide.
Thus, $\tau$ is an element of the finite set $A_\lambda$ by Cor.~\ref{cor:interval-finite-pell-bounds}.
\end{proof}

Furthermore, we can identify those Pell bounds which uniquely represent submaximal curves:

\begin{proposition}\label{prop:crit-pell-seshadri}
Let $L_\lambda$ be an ample $\Q$-line bundle with $\sqrt{L_\lambda^2}\notin\Q$ and let $A_\lambda$ be as in Prop.~\ref{prop:crit-pell-irred2}.
Then the following conditions are equivalent:
\begin{itemize}\compact
\item[\rm(i)]
   The Pell bound $\pi_\lambda$
   coincides with $\ell_C$ for some submaximal irreducible curve $C$.
\item[\rm(ii)]
   Every Pell bound $\pi_\mu$ with $\mu\in A_\lambda\setminus\set{\lambda}$ satisfies $\pi_\lambda(\lambda)<\pi_\mu(\lambda)$.
\end{itemize}
\end{proposition}

\begin{proof}
Assume that $\pi_\lambda=\ell_C$, and let $\pi_\mu$ be any Pell bound with $\mu\neq\lambda$.
Prop.~\ref{prop:irred-pell-bound} shows that the unique representative
of $\O_X(C)$ in $\mathcal N(X)$ is $L_\lambda$.
As $C$ computes the Seshadri constant of $L_\lambda$, we have $\pi_\lambda(\lambda)\leq \pi_\mu(\lambda)$.
Thus, we have to show that equality does not occur for $\lambda\neq\mu$.

Assume that $\pi_\lambda(\lambda)=\pi_\mu(\lambda)$ holds.
We will show that this implies $\lambda=\mu$.
By Prop.~\ref{prop:sesh-curve-1-submaximal} the submaximal curve $C$ computes the Seshadri constant in an open neighborhood $U$ at $L_\lambda$, hence
\be
\pi_\lambda(t)\leq \pi_\mu(t)\qquad \mbox{ for } t\in U\,.
\ee
This implies that the linear functions $\pi_{\lambda}$ and $\pi_\mu$ coincide, since otherwise we would have $\pi_{\lambda}(t)<\pi_\mu(t)$ for either $t<\lambda$ or $t>\lambda$, which is impossible because $\pi_{\lambda}$ computes the Seshadri function locally.
But by Prop.~\ref{prop:irred-pell-bound} the linear function $\ell_C$ only coincides with the Pell bound $\pi_\lambda$ and, thus, we have $\lambda=\mu$.

For the other implication, we argue as in the proof of Prop.~\ref{prop:crit-pell-irred2}:
For every
Seshadri curve $C$ of $L_\lambda$ there is a unique $\pi_\tau$ with $\ell_C=\pi_\tau$ and $\tau\in A_\lambda$.
Since $C$ computes the Seshadri constant of $L_\lambda$, the Pell bound $\pi_\tau$ computes the Seshadri constant in $\lambda$.
In particular, we have $\pi_\tau(\lambda)\leq \pi_\lambda(\lambda)$.
Since by assumption $\pi_\lambda(\lambda)<\pi_\mu(\lambda)$ for $\mu\ne\lambda$, we conclude that $\tau=\lambda$,
and therefore $\pi_\lambda=\ell_C$.
\end{proof}

So far, the assumption $\sqrt{L_\lambda^2}\notin\Q$ was crucial for our arguments since they depended on the existence of Pell divisors.
We will now show that the Seshadri constant can in fact be effectively computed for \textit{any} ample $\Q$-line bundle. This will complete the proof of Theorem~\ref{introthm:algorithm} stated in the introduction.

\begin{theorem}\label{thm:algorithm}
   There is an algorithm that
   computes the Seshadri constant
   of every given ample line bundle
   on
   principally polarized
   abelian surfaces with real multiplication.
\end{theorem}
\begin{proof}
If $L_\lambda$ is a $\Q$-line bundle such that $\sqrt{L_\lambda^2}\notin\Q$, then the assertion follows from the fact, that the set $A_\lambda$ from Prop.~\ref{prop:crit-pell-irred2} is finite.
Suppose then that $\sqrt{L_\lambda^2}\in\Q$.
We will construct a theoretical interval around $L_\lambda$, on which every Seshadri curve of $L_\lambda$ must be submaximal, if $\eps(L_\lambda)<\sqrt{L_\lambda^2}$.
By Cor.~\ref{cor:interval-finite-pell-bounds}, there are only finitely many Pell bounds on this interval, and
$\eps(L_\lambda)$ is the minimum of those.

Assume that the Seshadri constant satisfies $\varepsilon(L_\lambda)<\sqrt{L_\lambda^2}$.
Let $\lambda=\frac{p}{q}$ be a coprime representation.
Then, $L:=qL_\lambda$ is a primitive $\Z$-line bundle. Denote by $C$ any Seshadri curve of $L_\lambda$.
As explained in Sect.~\ref{sec:pell-curves-abelian}, there exists an effective Divisor $D\in |2L|^+$ such that $D$ satisfies ${D\cdot L}/{\mult_0(D)}\leq \sqrt{L^2}$.
As $C$ is the Seshadri curve of $L$, $C$ is a component of $D$ by \cite[Lemma~6.2]{Bauer:sesh-alg-sf}.
It follows that the intersection-number $C\cdot L$ is bounded by $D\cdot L=2L^2$.
As a consequence, the Seshadri constant can only take certain rational values:
\be
\varepsilon(L)\in \set{1\leq \frac{a}{b} < \sqrt{L^2} \with 1\leq a\leq 2L^2, b\in\N }\,.
\ee
Therefore, we find that the Seshadri constant is at most
\be
\varepsilon(L)\leq \frac{2L^2-1}{2\sqrt{L^2}}\,.
\ee
For the construction of the interval, we will give a lower and an upper bound for the slope of the linear function $\ell_C$.
To this end, we will chose any two rational numbers $\mu_i\in \mathcal N(X)$ with $\mu_1<\lambda<\mu_2$ and $\sqrt{L_{\mu_i}^2}\notin\Q$ for $i=1,2$.
Next, we compute a Seshadri curve
$C_i$ of $L_{\mu_i}$ using Prop.~\ref{prop:crit-pell-seshadri}.
We denote by $m_i$ the slope of the linear function $\ell_{C_i}$.
As the Seshadri function is a concave function, the slope $m$ of the linear function $\ell_C$ is bounded, $m_2\leq m\leq m_1$.
Let $r_i$ be the linear function passing through the point $(\lambda,({2L^2-1})/({2\sqrt{L^2}}))$ with slope $m_i$.
Then the function $u(t)=\max\set{r_1(t),r_2(t)}$ is an upper bound for $\ell_C$, since we have
\be
\ell_C(t) \leq r_1(t) \quad \mbox{ for } \lambda\leq t \qquad \mbox{ and } \qquad  \ell_C(s) \leq r_2(s) \quad \mbox{ for } s\leq \lambda\,.
\ee
Denote by $I$ the submaximality interval of $u$.
It follows that $C$ has to be submaximal on $I$, and so we have constructed a computable interval on which $C$ is submaximal.
Now, by following the same argument from Prop.~\ref{prop:crit-pell-irred2} we can compute the Seshadri constant of $L_\lambda$ by taking the minimum of  Pell bounds in $\lambda$, which are submaximal on $I$.
Clearly, if none of these Pell bounds are submaximal in $\lambda$, then the Seshadri constant satisfies $\varepsilon(L_\lambda)=\sqrt{L_\lambda^2}$.
\end{proof}

\begin{remark}
In the proof of Theorem~\ref{thm:algorithm} we have shown how to algorithmically distinguish the cases $\varepsilon(L_\lambda^2)<\sqrt{L_\lambda}$ and $\varepsilon(L_\lambda)=\sqrt{L_\lambda^2}$.
Both cases do, in fact, occur for line bundles $L$ with $\sqrt{L^2}\in\Q$:
Consider a principally polarized abelian surface $X$ with $\End(X)=\Z[\sqrt{2}]$.
Then the line bundle $L=2L_0+L_\infty$ satisfies $\varepsilon(L)=\sqrt{L^2}=2$, whereas the line bundle $L'=58L_0+L_\infty$ satisfies $\varepsilon(L')<\sqrt{L'^2}$, since the Seshadri curve $C$ of $L_0$ is also submaximal for $L'$.
\end{remark}

It is an important consequence of Thm.~\ref{thm:algorithm}
that the Seshadri function depends
only on the endomorphism ring of $X$, but not
on the isomorphism class of the surface:

\begin{theorem}\label{thm:coinciding-seshadri-functions}
Let $X$ and $Y$ be (not necessarily isomorphic) principally polarized abelian surfaces with real multiplication with $\End(X)\cong\End(Y)$.
Then their Seshadri functions coincide in the following sense:
Choosing suitable bases of the Néron-Severi groups $\NS(X)$ and $\NS(Y)$ 
yields an isomorphism $\Nef(X)\isom\Nef(Y)$, under which we have $\varepsilon_X=\varepsilon_Y$.
\end{theorem}
This implies Corollary~\ref{introcor:Endomorphism} stated in the introduction.
\begin{proof}
The proof of Thm.~\ref{thm:algorithm} shows that the numerical data
that enters the computation of the Seshadri functions stems
from the endomorphism ring. Therefore, an isometry of $\NS(X)$ that leaves
the ample cone invariant also leaves the Seshadri function invariant.
\end{proof}

\section{Fundamental cone and sample plots for Seshadri functions}\label{sec:sample-plots}
We will now determine the subgroup $G\subset\Aut(\NS(X))$ of isometries with respect to the intersection product that leave the Seshadri function on $\Amp(X)$ invariant.
This group gives rise to a decomposition of the ample cone into subcones on which $G$ acts transitively.

With respect to the basis $(L_0,L_\infty)$,
an automorphism $\varphi\in \Aut(\NS(X))$ is given by a matrix
$$M_\varphi=\matr{\alpha & \beta \\
      \gamma  & \delta}\in \mathrm {GL} _{2}(\Z)\,.$$
By Thm.~\ref{thm:coinciding-seshadri-functions} the Seshadri function remains invariant under the automorphism $M_\varphi$ if it is an isometry of $\NS(X)$
and additionally leaves the ample cone invariant.
These conditions can be expressed by:
\begin{itemize}\compact
\item[(i)] $L_0^2=(\alpha L_0 + \gamma L_\infty)^2$\,,
\item[(ii)] $L_\infty^2=(\beta L_0 + \delta L_\infty)^2$\,,
\item[(iii)] $L_0\cdot L_\infty=(\alpha L_0 + \gamma L_\infty)\cdot (\beta L_0 + \delta L_\infty)$\,,
\item[(iv)] $\alpha > 0$\,.
\end{itemize}
The conditions (i)--(iii) are equivalent to $\varphi$ being an isometry, whereas condition (vi) ensures that the ample cone is left invariant.

In the case of $\End(X)=\Z[\sqrt{e}]$ we find by solving (i)--(iv) that $M_\varphi$ is of the form
$$\matr{\alpha & e\,\beta \\
      \beta  & \alpha} \mbox{ or } \matr{\alpha & -e\,\beta \\
      \beta  & -\alpha}\,,\quad \mbox{ with } \alpha>0 \mbox{ and } \alpha^2 - e\beta^2=1\,.$$
Since any other Pell solution of $x^2-ey^2=1$ is generated by the minimal solution $(\alpha_0,\beta_0)$, the group $G$ is generated by
$$\varphi_0:=\matr{\alpha_0 & e\,\beta_0 \\
      \beta_0  & \alpha_0} \mbox{ and } \tau:=\matr{1 & 0 \\
      0  & -1}\,.$$
A line bundle $L=aL_0+bL_\infty$ is a principal polarization if and only if $(a,b)$ is a solution of Pell's equations $x^2-ey^2=1$ with $a>0$.
Therefore, we can express every principal polarization by $L_k:=x_kL_0+y_kL_\infty$, where $(x_k,y_k)$ satisfies
$$\matr{x_k \\
      y_k }=\matr{\alpha_0 & e\,\beta_0 \\
      \beta_0  & \alpha_0}^k \matr{1 \\
      0 }=\varphi_0^k\matr{1 \\
      0 }\,\qquad k\in\Z\,.$$
So we have $\varphi_0(L_k)=L_{k+1}$.
Next, we consider for $k\in\Z$ the subcone $\mathcal{D}_k\subset \Amp(X)$ generated by $L_k$ and $L_{k+1}$.
We have $\varphi_0^k(\mathcal{D}_0)=\mathcal{D}_k$.
Additionally, by also considering the automorphism $\tau$ we can further divide the subcone $\mathcal{D}_0$ into two subcones $\mathcal{D}_{0,1}$ and $\mathcal{D}_{0,2}$ as follows:
The automorphism $\varphi_0\circ \tau$
is of order two and maps the cone $\mathcal{D}_0$ onto itself.
The line bundle $L':=e\beta_0L_0+(\alpha_0-1)L_\infty$, which satisfies $\varphi_0\circ \tau(L')=L'$,
divides the subcone $\mathcal{D}_0$ into two subcones $\mathcal{D}_{0,1}$ and $\mathcal{D}_{0,2}$, where $\mathcal{D}_{0,1}$ is generated by $L_0$ and $L'$,
$$\mathcal{D}_{0,1}=\set{\lambda_1 L_0 + \lambda_2 L'\with \lambda_1,\lambda_2\ge 0}\,,$$
and $\mathcal{D}_{0,2}$ is generated by $L'$ and $L_1$.
The subcones $\mathcal{D}_{0,1}$ and $\mathcal{D}_{0,2}$ satisfy $\varphi_0\circ \tau(\mathcal{D}_{0,1})=\mathcal{D}_{0,2}$ and, again by construction of $\tau$, the Seshadri constants remain invariant.
We call $\mathcal{D}_{0,1}$
the \emph{fundamental cone} of $\Amp(X)$.
This cone corresponds to the interval $[0,\tfrac{\alpha_0-1}{e\beta_0}]$ in $\mathcal N(X)$.
The decomposition of the subcone $\mathcal{D}_0$ into $\mathcal{D}_{0,1}$ and $\mathcal{D}_{0,2}$ extends via $\varphi_0^k$ to every subcone $\mathcal{D}_k$.
After renumbering we obtain a decomposition of $\Amp(X)$ into subcones $\mathcal{C}_k$ with $\mathcal{C}_0=\mathcal{D}_{0,1}$.

We now deal with the
case of $\End(X)=\Z[\tfrac 12 + \tfrac 12 \sqrt{e}]$. In this case we find that $M_\varphi$ is given by
$$\matr{\alpha & \tfrac{e-1}{4}\,\beta \\
      \beta  & \alpha+\beta}\mbox{ or } \matr{\alpha & \alpha-\tfrac{e-1}{4}\,\beta \\
      \beta  & -\alpha}\,,\quad \mbox{ with } \alpha>0 \mbox{ and } \alpha^2 + \alpha\beta- \tfrac{e-1}{4}\beta^2=1\,.$$
Note that we have the following bijection
\be
\set{(x,y)\in\Z^2\,\vrule\, x^2-ey^2=4}
 &\stackrel{\displaystyle\sim}{\longrightarrow}&
\set{(x,y)\in\Z^2 \,\vrule\, x^2+xy-\tfrac{e-1}{4}y^2=1}\\
(x,y)&\mapsto &(\tfrac{x-y}{2},y)\,.
\ee
By \cite[Prop.~6.3.16]{Cohen:Number Theory} the set of solutions for the Pell-type equation $x^2-ey^2=4$ can be expressed through a minimal solution $(x_0,y_0)$ (we may assume $x_0>y_0>0$) as follows:
$$\set{\pm\frac{1}{2^k}\matr{x_0 & ey_0 \\
      y_0  & x_0}^k \matr{2 \\
      0 }\with  k\in\Z\,}.
$$
With some calculation, the set of solutions of $\alpha^2+\alpha\beta-\tfrac{e-1}{4}\beta^2=1$ can be determined as
$$\set{\pm\matr{\alpha_0 & \tfrac{e-1}{4}\beta_0 \\
      \beta_0  & \alpha_0+\beta_0}^k \matr{1 \\
      0 }\with k\in\Z},
$$
where $(\alpha_0,\beta_0):=(\tfrac{x_0-y_0}{2},y_0)$ and, hence, the group $G$ is generated by
$$\psi_0:=\matr{\alpha_0 & \tfrac{e-1}{4}\beta_0 \\
      \beta_0  & \alpha_0+\beta_0} \mbox{ and } \sigma:=\matr{1 & 1 \\
      0  & -1}\,.
$$
Using the exact same argument as before, we get a decomposition of the ample cone:
We can express every principal polarization by $L_k:=x_kL_0+y_kL_\infty$ with $(x_k,y_k)=\psi_0^k(1,0)$.
The subcones $\mathcal{D}_k$ generated by $L_k$ and $L_{k+1}$ satisfy $\psi_0^k(\mathcal{D}_0)=\mathcal{D}_k$.
Furthermore, $\psi_0\circ \sigma$ divides the subcone $\mathcal{D}_0$ into two subcones $\mathcal{D}_{0,1}$ and $\mathcal{D}_{0,2}$, which are generated by $L_0$ and $L':=(\alpha_0+1)L_0+\beta_0L_\infty$ and, respectively, $L'$ and $L_1$.
In this case, the fundamental cone $\mathcal{D}_{0,1}$ corresponds to the interval $[0,\tfrac{\beta_0}{\alpha_0+1}]$ in $\mathcal N(X)$.

The considerations above prove Theorem~\ref{introthm:cone} stated in the introduction.

\bigskip\noindent
   We now
   provide some sample plots in order to illustrate the behavior
   of Seshadri functions.
   Concretely, we compute for fixed $e$
   all Seshadri curves $C=qL_0+pL_\infty$ with $q\leq 3.000$
   that are contained in $\mathcal{C}_0$.
   From this set of curves we derive further Seshadri curves by
   applying the automorphisms in $G$.
   In the pictures, the dotted lines indicate the fundamental interval from which the complete Seshadri function can be computed by Thm.~\ref{introthm:cone}.

   The values of $e$ in figures~\ref{fig:2}--\ref{fig:33-2} are chosen in such a way
   that they illustrate different kinds of behavior:
   In the case of $\Z[\sqrt{2}]$ there exist $\Q$-line bundles $L_\lambda$ with $\varepsilon(L_\lambda)=\sqrt{L_\lambda^2}\in\Q$ whereas in the case of $\Z[\sqrt{5}]$ no such bundles exist.
   These line bundles generate \engqq{gaps} in the graph, because they do not give rise to a linear segment.
   In fact, at each of these gaps there are infinitely many linear segments which converge from both sides.
   In the plots for the case $\End(X)=\Z[\tfrac12+\tfrac12 \sqrt{e}]$ for $e=5$ and $33$
   the ample cone is not symmetric at $0$ in the case of $\End(X)=\Z[\tfrac12+\tfrac12 \sqrt{e}]$.

\begin{figure}[h]
\begin{tikzpicture}[x=9.0cm,y=2.5cm]
\draw[color=black] (0.7,1.5) node[above,draw] {$\Z[\sqrt{2}]$};

\draw[->,color=black] (-0.8,0.) -- (0.8,0.) node[right] {$t$};
\draw[->,color=black] (0,0) -- (0,1.6) node[above] {$\varepsilon(L_t)$};
\foreach \x in {-0.7,-0.6,-0.5,-0.4,-0.3,-0.2,-0.1,0,0.1,0.2,0.3,0.4,0.5,0.6,0.7}
\draw[shift={(\x,0)},color=black] (0pt,2pt) -- (0pt,-2pt) node[below] {\footnotesize $\x$};
\foreach \y in {0.2,0.4,0.6,0.8,1.,1.2,1.4}
\draw[shift={(0,\y)},color=black] (2pt,0pt) -- (-2pt,0pt) node[left] {\footnotesize $\y$};

\draw[dotted](0.5,0)--(0.5,4/3);
\input{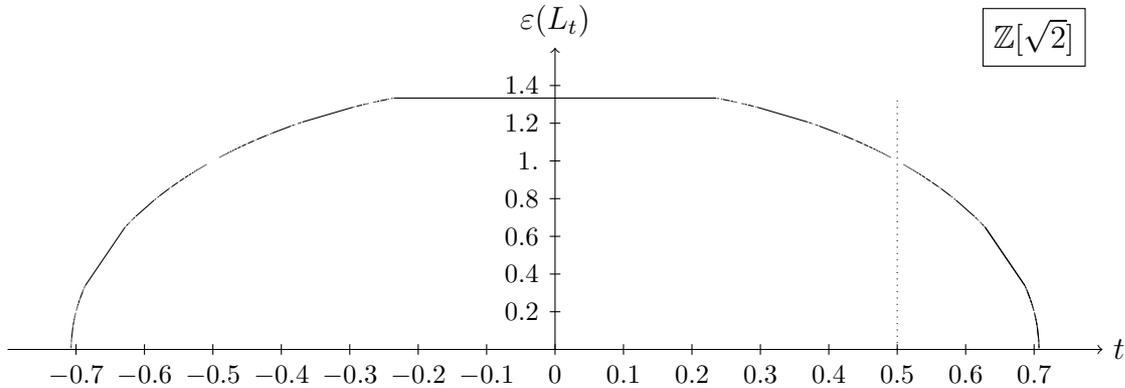}
\end{tikzpicture}
\caption{\label{fig:2}The Seshadri function of an abelian surface with real multiplication in $\Z[\sqrt{2}]$.}
\end{figure}

\begin{figure}[h]
\begin{tikzpicture}[x=12.0cm,y=2.5cm]
\draw[color=black] (0.55,1.5) node[above,draw] {$\Z[\sqrt{5}]$};

\draw[->,color=black] (-0.6,0.) -- (0.6,0.) node[right] {$t$};
\draw[->,color=black] (0,0) -- (0,1.6) node[above] {$\varepsilon(L_t)$};
\foreach \x in {-0.5,-0.4,-0.3,-0.2,-0.1,0,0.1,0.2,0.3,0.4,0.5}
\draw[shift={(\x,0)},color=black] (0pt,2pt) -- (0pt,-2pt) node[below] {\footnotesize $\x$};
\foreach \y in {0.2,0.4,0.6,0.8,1.,1.2,1.4}
\draw[shift={(0,\y)},color=black] (2pt,0pt) -- (-2pt,0pt) node[left] {\footnotesize $\y$};

\draw[dotted](0.4,0)--(0.4,4/3);
\input{pointlist-e-5.tex}
\end{tikzpicture}
\caption{\label{fig:5}The Seshadri function of an abelian surface with real multiplication in $\Z[\sqrt{5}]$.}
\end{figure}

\begin{figure}[h]
\begin{tikzpicture}[x=6cm,y=2.5cm]
\draw[color=black] (1.6,1.7) node[above,draw] {$\Z[\tfrac12 + \tfrac12\sqrt{5}]$};

\draw[->,color=black] (-0.7,0.) -- (1.7,0.) node[right] {$t$};
\draw[->,color=black] (0,0) -- (0,1.9) node[above] {$\varepsilon(L_t)$};
\foreach \x in {-0.6,-0.4,-0.2,0,0.2,0.4,0.6,0.8,1,1.2,1.4,1.6}
\draw[shift={(\x,0)},color=black] (0pt,2pt) -- (0pt,-2pt) node[below] {\footnotesize $\x$};
\foreach \y in {0.2,0.4,0.6,0.8,1.,1.2,1.4,1.6,1.8}
\draw[shift={(0,\y)},color=black] (2pt,0pt) -- (-2pt,0pt) node[left] {\footnotesize $\y$};

\draw[dotted](0.5,0)--(0.5,1.7);
\input{pointlist-e-5-h.tex}
\end{tikzpicture}
\caption{\label{fig:5-2}The Seshadri function of an abelian surface with real multiplication in $\Z[\tfrac12 + \tfrac12\sqrt{5}]$.}
\end{figure}

\begin{figure}[h]
\begin{tikzpicture}[x=18cm,y=2.5cm]
\draw[color=black] (.4,1.7) node[above,draw] {$\Z[\tfrac12 + \tfrac12\sqrt{33}]$};

\draw[->,color=black] (-0.325,0.) -- (0.45,0) node[right] {$t$};
\draw[->,color=black] (0,0) -- (0,1.9) node[above] {$\varepsilon(L_t)$};
\foreach \x in {-0.3,-0.2,-0.1,0,0.2,0.4,0.3,0.1}
\draw[shift={(\x,0)},color=black] (0pt,2pt) -- (0pt,-2pt) node[below] {\footnotesize $\x$};
\foreach \y in {0.2,0.4,0.6,0.8,1.,1.2,1.4,1.6,1.8}
\draw[shift={(0,\y)},color=black] (2pt,0pt) -- (-2pt,0pt) node[left] {\footnotesize $\y$};

\draw[dotted](2/5,0)--(2/5,1.5);
\input{pointlist-e-33-h.tex}
\end{tikzpicture}
\caption{\label{fig:33-2}The Seshadri function of an abelian surface with real multiplication in $\Z[\tfrac12 + \tfrac12\sqrt{33}]$.}
\end{figure}

The Seshadri function for $e=5$ consists only of linear segments which by Thm.~\ref{thm:one-submax-curve-lb} are never adjacent to each other.
In the case of $e=33$ there exist line bundles with two submaximal curves, e.g., at $t=0.37$.
In fact,
calculations show that there are chains of linear segments which overlap.
It should also be noted that the size of the fundamental interval depends heavily on the minimal solution of $x^2-ey^2=1$ or, respectively, $x^2+xy-\tfrac{e-1}{4}x^2=1$:
In the first three cases the minimal solutions are small, which leads to a small fundamental interval.
However,
experience with further examples has shown that the limit of the fundamental interval can be arbitrarily close to the interval limit of $\mathcal N(X)$.

\section{Distinguishing the cases of one and two submaximal curves}\label{sec:2-submaximal}

In this section we derive a method that allows one to distinguish whether all line bundles
on $X$ have at most one submaximal curve or if there exists a line bundle which has 2 submaximal curves.
By Thm.~\ref{thm:one-submax-curve-lb} we already know that there are infinitely many cases for $\End(X)=\Z[\tfrac 12 + \tfrac 12 \sqrt{e}]$, where every line bundle has at most one submaximal curve.
We will show that the case with two submaximal curves also appears infinitely many times.

\begin{proposition}\label{prop:crit-overlapping-curve}
There exists a line bundle on $X$ that has two submaximal curves if and only if there exist two Pell bounds $\pi_\lambda$ and $\pi_\mu$ such that the following two conditions are met:
\begin{enumerate}\compact
\item[(i)] Their submaximality intervals $J_\lambda$ and $J_\mu$ intersect and one is not contained in the other.
\item[(ii)] There does not exist
a Pell bound $\pi_\tau$ such that the submaximality interval $J_\tau$ contains  $J_\lambda\cup J_\mu$.
\end{enumerate}
\end{proposition}

\begin{proof}
Assume that there exists a line bundle $L$ on $X$ with two submaximal curves $C_1$ and $C_2$.
The linear functions $\ell_{C_1}$ and $\ell_{C_2}$ are Pell bounds by Prop.~\ref{prop:irred-pell-bound}, and their submaximality intervals $I_{C_1}$ and $I_{C_2}$ must intersect, because $C_1$ and $C_2$ are both $L$-submaximal.
By Lemma~\ref{lem:reducible-criterion} one submaximality interval can not be contained in the other.
Assume that there exists a Pell bound $\pi_\tau$ such that $I_{C_1}\cap I_{C_2} \subset J_\tau$.
Then any Pell divisor $P$ of $L_\tau$ is submaximal on $I_{C_1}\cap I_{C_2}$, and therefore $C_1$ and $C_2$ are reducible by Lemma~\ref{lem:reducible-criterion}, a contradiction.

Suppose now that there exist two Pell bounds $\pi_\lambda$ and $\pi_\mu$ such that (i) and (ii) holds.
The Pell bounds yields an upper bound for the Seshadri function in $J_\lambda\cup J_\mu$: We have
\be
\varepsilon(t)\leq \min\{\pi_\lambda(t),\pi_\mu(t)\}< \sqrt{L_t^2}
\righttext{for $t\in J_\lambda\cup J_\mu$.}
\ee
Let $C_1$ be a Seshadri curve
for a line bundle $L_{t_1}$ with $t_1\in J_\lambda\cup J_\mu$.
Due to (ii) the submaximality interval $I_{C_1}$ of $C_1$ cannot cover the complete interval $J_\lambda\cup J_\mu$.
Therefore, by continuity
there exists a $t_2\in (J_\lambda\cup J_\mu) \cap I_{C_1}$ such that
\be
\varepsilon(L_{t_2})\leq \min\{\pi_\lambda(t_2),\pi_\mu(t_2)\} < \frac{C_1\cdot L_{t_2}}{\mult_0 C_1} < \sqrt{L_{t_2}^2}\,,
\ee
i.e., $C_1$ is submaximal for $L_{t_2}$ but does not compute its Seshadri constant.
But the Seshadri constant of $L_{t_2}$ is computed by a curve, and thus there exists for $L_{t_2}$ another submaximal curve $C_2$ that computes the Seshadri constant.
It follows that $L_{t_2}$ has two submaximal curves.
\end{proof}

The criterion in Prop.~\ref{prop:crit-overlapping-curve} provides us with a
numerical method to search for line bundles with two submaximal curves:
First, we search for Pell bounds whose submaximality intervals intersect.
After that, one checks by using Prop.~\ref{prop:crit-pell-irred2} whether there exists another Pell bound which contains both intervals.
Using computer-assisted computation, this yields the following:

\begin{proposition}\label{prop:overlapping-curves-50.000}
Suppose that $\End(X)=\Z[\tfrac12 + \tfrac12 \sqrt{e}]$ for a non-square integer $e$ with $0<e\le 25.000$, such that we have $e\equiv 1$ modulo $4$ and $e$ \textit{does not} have a prime factor $p$ with $p\equiv 5$ or $7$ modulo $8$.
Then there exists a line bundle on $X$ with two submaximal curves.
\end{proposition}

Theorem~\ref{thm:one-submax-curve-lb} and the previous proposition suggest the following conjecture:

\begin{conjecture}\label{conj:cases}
Let $L$ be any ample $\Q$-line bundle on $X$.
Then there exists at most one irreducible curve $C$ that is submaximal for $L$ if and only if $\End(X)$ satisfies either
\begin{itemize}\compact
\item $\End(X)=\Z[\sqrt{e}]$ for a non-square integer $e>0$, or
\item $\End(X)=\Z[\tfrac12 + \tfrac12 \sqrt{e}]$ for a non-square integer $e>0$, such that $e\equiv 1$ modulo $4$ and $e$ has a prime factor $p$ with $p\equiv 5$ or $7$ modulo $8$.
\end{itemize}
\end{conjecture}

\begin{remark}
One can show by applying well-known results on quadratic residues and binary quadratic forms
(see e.g. \cite[Prop.~2.2.4]{Cohen:Number Theory} and \cite[Lemma~2.5]{Cox:Number Theory})
that the following conditions are equivalent for a non-square integer $e$ with $e\equiv 1$ modulo $4$:
\begin{enumerate}\compact
\item[(i)] $e$ does not have any prime factor $p$ with $p\equiv 5$ or $7$ modulo $8$.
\item[(ii)] $-2$ is a quadratic residue modulo $e$.
\item[(iii)] $e=A^2+8B^2$ for some $A,B\in \N$ with $\gcd(A,B)=1$.
\end{enumerate}
\end{remark}

Finally, we will show that the case with two submaximal curves occurs infinitely often.

\begin{proposition}\label{prop:sequence-overlapping}
Let $e_n:=1+8n^2$.
If $e_n$ is not a perfect square, then
every principally polarized abelian surface with $\End(X)=\Z[\tfrac12 + \tfrac12\sqrt{e_n}]$ has a line bundle with two submaximal curves.
\end{proposition}

\begin{proof}
Consider the ample line bundles $L=2nL_0+L_\infty$ and $L'=(2n-1)L_0+L_\infty$.
The Pell solution of $x^2-L^2y^2=1$ is given by $(2n+1,1)$, and the Pell solution of $x^2-L'^2y^2=1$ is $(2n-1,1)$.
Hence, the submaximality intervals of the corresponding Pell bounds $\pi_\frac{1}{2n}$ and $\pi_\frac{1}{2n-1}$ are given by
\be
J_\frac{1}{2n}=\left(\frac{16n^3+(2n+1) (1-\sqrt{8n^2+1})}{8n^2(4n^2+1)},\frac{16n^3+(2n+1)(1+\sqrt{8n^2+1})}{8n^2(4n^2+1)}\right)
\ee
and, respectively,
\be
J_\frac{1}{2n-1}=\left(\frac{2n+(2n-1)(8n^2-\sqrt{8n^2+1})}{32n^4-32n^3+16n^2-4n+1},\frac{2n+(2n-1)(8n^2+\sqrt{8n^2+1})}{32n^4-32n^3+16n^2-4n+1} \right)\,.
\ee
Explicit computations show that both Pell bounds $\pi_\frac{1}{2n}$ and $\pi_\frac{1}{2n-1}$ are submaximal at $\frac{2}{4n-1}$, and therefore their submaximality intervals intersect.

So the first condition of Prop.~\ref{prop:crit-overlapping-curve} is satisfied.
In order to conclude that a line bundle with two submaximal curves exists, it remains to show that
there does not exist another Pell bound $\pi_\lambda$ whose submaximality interval covers the interval $I := J_\frac{1}{2n}\cup J_\frac{1}{2n-1}$.
For this, we will derive an upper bound and a lower bound for the denominator $q$ of $\lambda=\frac{p}{q}$ which must be satisfied if the Pell bound $\pi_\lambda$ covers $I$.
As we will see, the upper and lower bound contradict each other and thus there cannot exist such a Pell bound.

\textit{Upper bound for $q$:}
The Pell bound $\pi_\lambda$ has to cover both submaximality intervals, i.e., the interval
\be
I=\left(\frac{16n^3+(2n+1)(1 -\sqrt{8n^2+1})}{8n^2(4n^2+1)},\frac{2n+(2n-1)(8n^2+\sqrt{8n^2+1})}{32n^4-32n^3+16n^2-4n+1} \right)\,.
\ee
One can show that the length of this interval is at least $(\sqrt{2}+1)/(4n^2)$,
and using Lemma~\ref{lem:interval-length} we derive the upper bound
\be
q\leq \frac{4\sqrt{11}n^2}{(\sqrt{2}+1)\sqrt{8n^2+1}}\leq \frac{35}{18}n\,.
\ee

\textit{Lower bound for $q$:}
First, we observe that the unique Pell bound $\pi_0$ is not submaximal for $L$ and $L'$ and, thus, we may assume that $\lambda\neq 0$, i.e. $p\neq 0$.
We obtain a preliminary lower bound for $q$ by taking into account that the line bundle $L_\lambda$ has to be ample, i.e.,
\be
L_\lambda^2=2+\frac{2p}{q} -  \frac{4n^2p^2}{q^2} >0\,,
\ee
and, thus,
\be
q\geq \frac{p}{2}(\sqrt{8n^2+1}-1)\geq \sqrt{2}\,(n-1)\,.
\ee
Unfortunately, this lower bound yields no contradiction with our upper bound.
However, it provides us with a method to refine the lower bound.
Using the computation from Lemma~\ref{lem:interval-length}, we find a maximal possible length for the submaximality interval $J_\lambda=(t_1,t_2)$ provided that $\sqrt{2}\,(n-1)\leq q \leq \frac{35}{18}n$:
\be
t_2 - t_1  < \frac{2\sqrt{2+\frac{2}{e_n-1}+\frac{1}{k^2q^2}} }{q\sqrt{e_n}}\leq \frac{2\sqrt{2+\frac{1}{4n^2}+\tfrac{1}{2\,(n-1)^2}} }{\sqrt{2}(n-1)\sqrt{1+8n^2}}
 <  \frac{\sqrt{2+\frac{1}{4n^2}+\tfrac{1}{2\,(n-1)^2}} }{2n(n-1)}\,.
\ee
This in turn, gives us an upper bound for $\lambda$, since the submaximality interval of $\pi_\lambda$ can cover at most $t_1-t_2$:
\be
\lambda\leq \frac{16n^3+2n+1 -(2n+1)\sqrt{8n^2+1}}{8n^2(4n^2+1)} + \frac{\sqrt{2+\frac{1}{4n^2}+\tfrac{1}{2\,(n-1)^2}} }{2n(n-1)}\,.
\ee
It follows that $\lambda\leq \frac{1}{2n-3}$ and, therefore, the denominator $q$ of $\lambda$ must be at least $2n-3$.

This shows that for $n\geq 55$ there cannot exist a Pell bound whose submaximality interval
covers $J_\frac{1}{2n}$ and $J_\frac{1}{2n-1}$.
Thus, the assertion follows for $n\geq 55$ from Prop.~\ref{prop:crit-overlapping-curve}.
The explicit computations from Prop.~\ref{prop:overlapping-curves-50.000} cover the remaining cases for $n\leq 54$.
\end{proof}

\begin{remark}
The case where $e_n$ is a square number, i.e., $e_n=1+8n^2=r^2$ for an integer $r\in \N$, is equivalent to the case where $(r,n)$ is a solution for the Pell equation $x^2-8y^2=1$, and hence there are infinitely many $n$ such that $e_n$ is not a square number.
\end{remark}



\footnotesize
   \bigskip
   Thomas Bauer,
   Fachbereich Mathematik und Informatik,
   Philipps-Universit\"at Marburg,
   Hans-Meerwein-Stra\ss e,
   D-35032 Marburg, Germany.

   \nopagebreak
   \textit{E-mail address:} \texttt{tbauer@mathematik.uni-marburg.de}

   \bigskip
   Maximilian Schmidt
   Fachbereich Mathematik und Informatik,
   Philipps-Universit\"at Marburg,
   Hans-Meerwein-Stra\ss e,
   D-35032 Marburg, Germany.

   \nopagebreak
   \textit{E-mail address:} \texttt{schmid4d@mathematik.uni-marburg.de}


\end{document}